\documentclass{article}
\usepackage[a4paper,margin=1.5in,footskip=0.25in]{geometry}
\usepackage{setspace}
\usepackage{amssymb}
\usepackage{amsmath}
\usepackage{amsthm}
\usepackage{stmaryrd}
\usepackage{graphicx}
\usepackage{proof}
\usepackage{enumerate}
\usepackage{appendix}
\usepackage{enumitem}
\usepackage{multicol}
\usepackage{mathtools}
\usepackage{hyperref}
\usepackage{xcolor}
\hypersetup{
	colorlinks=true,
	linkcolor=blue,
	citecolor=blue
}

\theoremstyle{plain}
\newtheorem{lemma}{Lemma}
\newtheorem{theorem}{Theorem}
\newtheorem{corollary}{Corollary}
\newtheorem{proposition}{Proposition}

\theoremstyle{definition}

\newtheorem{definition}{Definition}

\theoremstyle{remark}
\newtheorem{example}{Example}
\newtheorem{remark}{Remark}

\newcommand{\ACTMult}{!^m\nabla \mathrm{ACT}_\omega}
\newcommand{\ACT}{\mathrm{ACT}_\omega}
\newcommand{\Logic}{\mathrm{L}}

\newcommand{\PN}{\mathfrak{P}}

\newcommand{\SATOO}{\mathrm{SAT}_{<\omega^\omega}}
\newcommand{\SAT}{\mathrm{SAT}}

\newcommand{\Nat}{\mathbb{N}}
\newcommand{\PrFm}{\mathrm{Pr}}
\newcommand{\Fm}{\mathrm{Fm}}
\newcommand{\fm}{\mathit{fm}}
\newcommand{\ax}{\mathrm{ax}}
\newcommand{\eqdef}{\stackrel{\text{\tiny def}}{=}}
\newcommand{\fail}{\mathit{fail}}
\newcommand{\final}{\sigma}
\newcommand{\inp}{\mathit{inp}}
\newcommand{\go}{\mathit{go}}
\newcommand{\wait}{\mathit{wait}}

\newcommand{\okay}{\mathit{ok}}
\newcommand{\OKAY}{\mathit{OK}}

\newcommand{\Killer}{\mathcal{K}}
\newcommand{\Energy}{\mathcal{E}}
\newcommand{\energy}{\epsilon}

\newcommand{\HEnergy}{\widehat{\mathcal{E}}}
\newcommand{\Technical}{\mathrm{Tech}}

\newcommand{\Rule}{\mathrm{Rule}}
\newcommand{\bs}{\backslash}
\newcommand{\bang}{\mathop{!}}
\newcommand{\blank}{\lambda}

\newcommand{\res}{\mathrm{res}}
\newcommand{\Aug}{\mathrm{Aug}}
\newcommand{\hyp}{\mathrm{hyp}}

\newcommand{\Tape}{\mathit{Tp}}
\newcommand{\Input}{\mathit{In}}
\newcommand{\Output}{\mathit{Out}}
\newcommand{\Trans}{\delta}
\newcommand{\TM}{TM}
\newcommand{\SR}{\mathrm{SR}}
\newcommand{\sr}{\mathit{SR}}

\newcommand{\DerTree}{\mathfrak{D}}
\newcommand{\Tree}{\mathfrak{T}}

\newcommand{\Halt}{\mathrm{Halt}_0}
\newcommand{\HALT}{\mathrm{Halt}}

\newcommand{\Premise}{\mathrm{Prem}}
\newcommand{\Axiom}{\mathrm{Ax}}
\newcommand{\Der}{\mathrm{Der}}
\newcommand{\LogicOOO}{L_{\omega_1\omega}}

\DeclareMathOperator*{\bigdoublewedge}{\bigwedge\mkern-15mu\bigwedge}
\DeclareMathOperator*{\bigdoublevee}{\bigvee\mkern-15mu\bigvee}

\definecolor{highlight}{RGB}{255,210,210}
\definecolor{principal}{RGB}{255,255,210}
\definecolor{reduced}{RGB}{210,255,210}

\newcommand{\principal}[1]{\fbox{$#1$}}

\begin{document}
	
	\title{Hyperarithmetical Complexity of \\ Infinitary Action Logic with Multiplexing}
	
	\author
	{
		Tikhon Pshenitsyn
		\\
		\href{mailto:tpshenitsyn@mi-ras.ru}{tpshenitsyn@mi-ras.ru} 
		\\
		Steklov Mathematical Institute of Russian Academy of Sciences
		\\
		8 Gubkina St., Moscow 119991, Russia
	}

	\maketitle
	
	\begin{abstract}
		In 2023, Kuznetsov and Speranski introduced infinitary action logic with multiplexing $\ACTMult$ and proved that the derivability problem for it lies between the $\omega$ and $\omega^\omega$ levels of the hyperarithmetical hierarchy. We prove that this problem is $\Delta^0_{\omega^\omega}$-complete under Turing reductions. Namely, we show that it is recursively isomorphic to the satisfaction predicate for computable infinitary formulas of rank less than $\omega^\omega$ in the language of arithmetic. As a consequence we prove that the closure ordinal for $\ACTMult$ equals $\omega^\omega$. We also prove that the fragment of $\ACTMult$ where Kleene star is not allowed to be in the scope of the subexponential is $\Delta^0_{\omega^\omega}$-complete. Finally, we present a family of logics, which are fragments of $\ACTMult$, such that the complexity of the $k$-th logic lies between $\Delta^0_{\omega^k}$ and $\Delta^0_{\omega^{k+1}}$.
	\end{abstract}

	\section{Introduction}\label{sec_introduction}
	
	Infinitary action logic $\ACT$ is an axiomatization of $^\ast$-continuous residuated Kleene lattices, which is introduced in \cite{Palka07}. Effectively, $\ACT$ is a substructural logic extending the full Lambek calculus with the Kleene star operation. Infinitariness of $\ACT$ comes from the following $\omega$-rule for Kleene star, which has countably many premises:
	$$
	\infer[(\ast L_\omega)]{\Gamma, A^\ast, \Delta \vdash B}{\left( \Gamma, A^n, \Delta \vdash B \right)_{n<\omega}}
	$$
	Another operations one can enrich the Lambek calculus with are exponential and subexponential modalities of linear logic. Extensions of the Lambek calculus with such operations are now being studied by logicians and linguists who suggest using them e.g. for modelling anaphora and ellipsis in natural languages \cite{McPheatWSCT21}. Over and above that, in \cite{KuznetsovS22, KuznetsovS23}, exponential and subexponential modalities are added to infinitary action logic. The resulting logics are of interest from the point of view of their complexity. The derivability problem for infinitary action logic with the exponential has very high complexity, namely, it is $\Pi^1_1$-complete \cite{KuznetsovS22}. In \cite{Kuznetsov21}, a fragment of this logic is studied where Kleene star is not allowed to appear in
	the scope of the exponential. Its complexity is proved to be in $\Delta^1_1$, so it is hyperarithmetical; besides, derivability in this logic is $\Pi^0_2$-hard hence this is a logic of intermediate complexity between $\ACT$ and $\ACT$ with the exponential.
	
	Another extension of $\ACT$, which was introduced and studied in \cite{KuznetsovS23}, is \emph{infinitary action logic with multiplexing} $\ACTMult$. It contains a weaker subexponential modality admitting the so-called \emph{multiplexing} introduction rule (actually, there are countably many rules indexed by $n \in \Nat$):
	$$
	\infer[(\bang L_n)]{\Gamma, \bang A, \Delta \vdash B}{\Gamma, A^n, \Delta \vdash B}
	$$
	Apart from $\bang$ and $^\ast$, the logic $\ACTMult$ contains the modality $\nabla$ such that a formula of the form $\nabla A$ admits the permutation rule. Complexity of $\ACTMult$ is known to be hyperarithmetical; more precisely, it lies in the $\omega^\omega$ level of the hyperarithmetical hierarchy \cite{KuznetsovS23}. For the lower bound, this logic is known to be $\Delta^0_\omega$-hard because the first-order theory of the standard model of arithmetic is m-reducible to derivability in $\ACTMult$ \cite{KuznetsovS23}. However, the exact complexity of this logic is an open problem \cite{KuznetsovS23}. 
	
	In this paper, we solve this problem proving the following theorem:
	\begin{theorem}\label{th_main_complexity}
		The derivability problem for $\ACTMult$ is $\Delta^0_{\omega^\omega}$-complete under Turing reductions.
	\end{theorem}
	Namely, we show that the set of sequents derivable in $\ACTMult$ is recursively isomorphic to the satisfaction predicate in the standard model of arithmetic for the fragment of infinitary logic $\LogicOOO$ that consists of computable infinitary formulas of rank less than $\omega^\omega$ (Theorem \ref{th_main_1-equivalence}). This isomorphism can be viewed as a generalization of the m-reduction of the first-order theory of arithmetic to the derivability problem for $\ACTMult$ presented in \cite{KuznetsovS23}. Note that the structure of the hyperarithmetical hierarchy does not allow us to formulate the complexity result in terms of m-reducibility because there is no invariant definition of a $\Delta^0_{\omega^\omega}$-complete set under m-reductions if one uses Kleene's $\mathcal{O}$ notation \cite{Rogers67} (look, however, at Corollary \ref{corollary_main_m-reducibility}). As a corollary of $\Delta^0_{\omega^\omega}$-completeness of $\ACTMult$, we prove that the closure ordinal for this logic equals $\omega^\omega$ (finding its exact value was left open in \cite{KuznetsovS23}); see Section \ref{ssec_closure_ordinal} for details.
	
	The main part of the proof of Theorem \ref{th_main_complexity} is constructing a reduction of the satisfaction predicate to derivability in $\ACTMult$. It consists of writing a ``low-level program'' in the ``programming language'' $\ACTMult$ (Section \ref{ssec_construction}) which checks whether a given computable infinitary formula is true under a given assignment of variables in the standard model of arithmetic. Checking that the ``program'' behaves correctly requires certain proof analysis in $\ACTMult$ (Section \ref{ssec_BTA}). In Section \ref{ssec_ACTMult_minus}, we also present a variant of the reduction for the logic $\ACTMult^-$; in formulas of this logic, Kleene star does not appear in the scope of the multiplexing subexponential. In what follows, this restriction on formulas of $\ACTMult$ does not decrease the complexity of this logic. Contrast this with the same restriction for infinitary action logic with the exponential, which is known to decrease complexity of the derivability problem for this logic. 
	
	Apart from this, in Section \ref{ssec_family_of_fragments}, we introduce the family of fragments $\ACTMult^{(k)}$ of the logic $\ACTMult$ (formulas of $\ACTMult^{(k)}$ are those with $\{\ast,\bang\}$-depth being not greater than $k$) such that the complexity of $\ACTMult^{(k)}$ lies between $\Delta^0_{\omega^k}$ and $\Delta^0_{\omega^{k+1}}$. Hence we have countably many logics with non-decreasing complexities that tend to $\omega^\omega$.
	
	The achieved complexity results are interesting from the following point of view. According to Theorem \ref{th_main_complexity}, $\ACTMult$ is a propositional logic with the hyperarithmetical algorithmic complexity $\Delta^0_{\omega^\omega}$. To our best knowledge, no logical system studied in the literature has the same complexity. The only naturally formulated problem we have found that has the same complexity is the isomorphism problem for well-founded tree-automatic trees \cite{KartsowJM12}, which, however, is of different nature. 
	
	The paper is organized as follows. In Section \ref{sec_preliminaries}, we provide preliminary definitions. We define the logic $\ACTMult$ and survey its basic properties; there we also recall the main ideas behind the hyperarithmetical hierarchy. Section \ref{sec_main} presents the main results and their proofs. In Section \ref{sec_conclusion}, we conclude.

	\section{Preliminaries}\label{sec_preliminaries}
	
	In this section, we introduce all the basic machinery behind our main construction. Firstly, we present infinitary action logic with multiplexing and prove some its properties that we use in the following section. Secondly, we define relevant concepts of the computability theory; we introduce the hyperarithmetical hierarchy (using not Kleene's $\mathcal{O}$ but a nice polynomial notation) and also present definitions related to Turing machines which are needed to handle technicalities in the main construction.
	
	\subsection{Infinitary Action Logic with Multiplexing}\label{ssec_ACTMult}
	In \cite{KuznetsovS23}, the logic $\ACTMult$ is defined as a sequent calculus. Let $\PrFm$ be the set of primitive formulas. Formulas are built from primitive formulas and constants $\mathbf{0},\mathbf{1}$ using $\bs, /, \cdot, \wedge,\vee,\bang,^\ast$. The set of formulas is denoted by $\Fm$. A sequent is a structure of the form $A_1,\dotsc,A_n \vdash B$ where $n \ge 0$ and $A_i$, $B$ are formulas from $\Fm$. The axiom and the rules of $\ACTMult$ are presented below.
	$$
	\infer[(\ax)]{A \vdash A}{}
	$$
	$$
	\infer[(\bs L)]{\Gamma, \Pi, B \bs A, \Delta \vdash C}{\Gamma, A, \Delta \vdash C & \Pi \vdash B}
	\qquad
	\infer[(\bs R)]{\Pi \vdash B \bs A}{B, \Pi \vdash A}
	$$
	$$
	\infer[(/ L)]{\Gamma, A / B, \Pi, \Delta \vdash C}{\Gamma, A, \Delta \vdash C & \Pi \vdash B}
	\qquad
	\infer[(/ R)]{\Pi \vdash A / B}{\Pi, B \vdash A}
	$$
	$$
	\infer[(\cdot L)]{\Gamma, A \cdot B, \Delta \vdash C}{\Gamma, A, B, \Delta \vdash C}
	\qquad
	\infer[(\cdot R)]{\Gamma, \Delta \vdash A \cdot B}{\Gamma \vdash A & \Delta \vdash B}
	$$
	$$
	\infer[(\wedge L_i), i=1,2]{\Gamma, A_1 \wedge A_2, \Delta \vdash C}{\Gamma, A_i, \Delta \vdash C}
	\qquad
	\infer[(\wedge R)]{\Pi \vdash A_1 \wedge A_2}{\Pi \vdash A_1 & \Pi \vdash A_2}
	$$
	$$
	\infer[(\vee  L)]{\Gamma , A_1 \vee A_2, \Delta \vdash C}{\Gamma , A_1 , \Delta \vdash C & \Gamma , A_2, \Delta \vdash C}
	\qquad
	\infer[(\vee R_i), i=1,2]{\Pi \vdash A_1 \vee A_2}{\Pi \vdash A_i}
	$$
	$$
	\infer[(\mathbf{0} L)]{\Gamma, \mathbf{0}, \Delta \vdash C}{}
	\qquad
	\infer[(\mathbf{1} L)]{\Gamma , \mathbf{1}, \Delta \vdash C}{\Gamma , \Delta \vdash C}
	\qquad
	\infer[(\mathbf{1} R)]{\vdash \mathbf{1}}{}
	$$
	$$
	\infer[(\bang L_n), n < \omega]{\Gamma, \bang A, \Delta \vdash B}{\Gamma, A^n, \Delta \vdash B}
	\qquad
	\infer[(\bang R)]{\bang A \vdash \bang B}{A \vdash B}	
	$$
	$$
	\infer[(\ast L_\omega)]{\Gamma, A^\ast, \Delta \vdash B}{\left( \Gamma, A^n, \Delta \vdash B \right)_{n<\omega}}
	$$
	$$
	\infer[(\ast R_0)]{\vdash A^\ast}{}
	\qquad
	\infer[(\ast R_n), 0 < n < \omega]{\Pi_1,\dotsc,\Pi_n \vdash A^\ast}{\Pi_1 \vdash A & \dotsc & \Pi_n \vdash A}
	$$
	$$
	\infer[(\nabla L)]{\Gamma, \nabla A, \Delta \vdash B}{\Gamma, A, \Delta \vdash B}
	\qquad
	\infer[(\nabla R)]{\nabla A \vdash \nabla B}{A \vdash B}	
	$$
	$$
	\infer[(\nabla P_1)]{\Gamma, \nabla A, \Pi, \Delta \vdash B}{\Gamma, \Pi, \nabla A, \Delta \vdash B}
	\qquad
	\infer[(\nabla P_2)]{\Gamma, \Pi, \nabla A, \Delta \vdash B}{\Gamma, \nabla A, \Pi, \Delta \vdash B}
	$$
	
	When we consider a formula of the form $A_1 \bs \dotsc \bs A_m \bs B$ we assume the bracketing $A_1 \bs (\dotsc \bs (A_m \bs B)\dotsc)$.
	
	When we consider an instantiation of a rule of $\ACTMult$ (we also call it a rule application), we use the standard notion of a principal occurrence of a formula in the conclusion. The principal formula in a rule application is the distinguished occurrence of a formula in the conclusion of the rule as defined above. E.g. in the rule $(\bs L)$, $B \bs A$ from the conclusion is principal; in the rule $(\bs R)$, the succedent $B \bs A$ is principal; in $(\bang R)$, $\bang B$ is principal; etc. There are no principal formulas in axioms.
	
	The cut rule is admissible in $\ACTMult$ \cite[Theorem 4.1]{KuznetsovS23}:
	$$
	\infer[(\mathrm{cut})]{\Gamma,\Pi,\Delta \vdash C}{\Pi \vdash A & \Gamma,A, \Delta \vdash C}
	$$
	
	The following is a standard corollary of the cut rule admissibility.
	
	\begin{corollary}\label{corollary_reversible_rules}
		The rules $(\cdot L)$, $(\vee L)$ and $(\ast L_\omega)$ are reversible.
	\end{corollary}
	
	This means that, if we take an instantiation of any of these rules and if its conclusion is derivable, then so are all its premises. E.g. if we know that $\Gamma, A^\ast, \Delta \vdash B$ is derivable, then we apply the cut rule to it and to the derivable sequent $A^n \vdash A^\ast$ thus concluding that the sequent $\Gamma, A^n, \Delta \vdash B$ is derivable (which holds for any $n$).
	
	Another important feature of $\ACTMult$ is the following ranking function defined on formulas and sequents \cite[Section 3]{KuznetsovS23}:
	
	\begin{itemize}
		\item $\rho(p) = 1$ for $p \in \PrFm$; \qquad $\rho(\mathbf{0}) = \rho(\mathbf{1}) = 1$;
		\item $\rho(A \bs B) = \rho(B/A) = \rho(A \cdot B) = \rho(A\wedge B) = \rho(A \vee B) = \rho(A) \oplus \rho(B) \oplus 1$ (here $\alpha \oplus \beta$ is the natural sum of ordinals\footnote{The Hessenberg natural sum is defined as follows: $(\omega^n\cdot a_n + \dotsc + \omega^0\cdot a_0) \oplus (\omega^n\cdot b_n + \dotsc + \omega^0\cdot b_0) = \omega^n\cdot (a_n+b_n) + \dotsc + \omega^0\cdot (a_0+b_0)$ for $a_i,b_i \in \Nat$.});
		\item $\rho(A^\ast) = \rho(\bang A) = \rho(A) \cdot \omega + 1$; \qquad $\rho(\nabla A) = \rho(A) + 1$;
		\item $\rho(A_1,\dotsc,A_n \vdash B) = \rho(A_1) \oplus \dotsc \oplus \rho(A_n) \oplus \rho(B)$.
	\end{itemize}
	
	Note that $\rho(\Pi \vdash B) < \omega^\omega$ for any sequent $\Pi \vdash B$. 
	
	\begin{remark}\label{remark_rank_increases}
		It is proved in \cite[Lemma 3.1]{KuznetsovS23} that, for each rule of $\ACTMult$ except for $(\nabla P_i)$, the rank of each its premise is strictly less that the rank of its conclusion. Clearly, the rules $(\nabla P_i)$ violate this property since they only change the order of formulas in the antecedent. To handle the permutation rules $(\nabla P_i)$, one can reformulate $\ACTMult$ by introducing \emph{generalized rules} \cite[p. 264]{KuznetsovS23}. A generalized rule is a rule which is not $(\nabla P_i)$, followed by one application of a permutation rule, e.g.:
		$$
		\infer[(\nabla P_1)]
		{
			\nabla D, \Gamma, A / B, \Pi_1,\Pi_2, \Delta \vdash C
		}
		{
			\infer[(/ L)]{\Gamma, A / B, \Pi_1,\nabla D, \Pi_2, \Delta \vdash C}{\Gamma, A, \Delta \vdash C & \Pi_1,\nabla D, \Pi_2 \vdash B}
		}
		$$
		
		\noindent
		After replacing each rule of $\ACTMult$ by its generalized version, permutation rules themselves become unnecessary, and one can remove them. Doing this explicitly, however, would make the reasonings extremely cumbersome. If needed, we shall implicitly assume that $\ACTMult$ is defined using generalized rules.
	\end{remark}

	\subsection{Basic Derivations}\label{ssec_basic_derivations}
	
	From now on, let us consider the fragment of $\ACTMult$ without constants $\mathbf{0},\mathbf{1}$ and also without $/$, $\vee$. This is because we shall not use these constants and operations in the main construction. Further reasonings could be done in presence of $\mathbf{0}, \mathbf{1}, /, \vee$ as well but this would just make definitions and proofs pointlessly longer.
	
	Below we introduce the notion of a basic derivation, which can be viewed as a normal form of derivations in $\ACTMult$.

	\begin{definition}\label{def_basic}
		A derivation $\DerTree$ of a sequent in $\ACTMult$ is \emph{basic} if it satisfies the following properties:
		\begin{enumerate}
			\item\label{def_basic_bs_L} If an occurrence of the formula $p \bs A$ is principal in a rule application of $(\bs L)$ for $p \in \PrFm$, then this rule application must be of the form
			$$
			\infer[(\bs L)]{\Gamma, p, p \bs A, \Delta \vdash C}{\Gamma, A, \Delta \vdash C & p \vdash p}
			$$
			\item\label{def_basic_cdot_R} If an occurrence of the formula $p \cdot q$ is principal in a rule application of $(\cdot R)$ for $p,q \in \PrFm$, then this rule application must be of the form
			$$
			\infer[(\cdot R)]{p, q \vdash p \cdot q}{p \vdash p & q \vdash q}
			$$
			\item\label{def_basic_wedge_L} Assume that $A_1 \wedge A_2$ is a principal formula in an application of $(\wedge L_i)$ for some $i\in \{1,2\}$. Let $A_i = B \bs A$. Then the rule application of $(\wedge L_i)$ must be preceded by a rule application of $(\bs L)$ with $A_i$ being principal.
		\end{enumerate}
	\end{definition}
	
	\begin{lemma}\label{lemma_basic_derivation}
		If a sequent $\Theta \vdash D$ is derivable in $\ACTMult$, then there exists its basic derivation.
	\end{lemma}
	
	To explain how to transform any derivation of a given sequent into a basic one we introduce several operations on derivation trees.
	
	\begin{definition}
		An \emph{incomplete derivation tree} is a derivation tree $\Tree$ in $\ACTMult+\mathbb{S}$ for $\mathbb{S}$ being a set of sequents added to $\ACTMult$ as additional axioms. The sequent that stands in the root of this tree is denoted by $\res(\Tree)$.
	\end{definition}
	
	In other words, some of the leaves of an incomplete derivation tree $\Tree$ are not axioms of $\ACTMult$; let us call them \emph{hypotheses in $\Tree$}. Let us denote the set of hypotheses in $\Tree$ by $\hyp(\Tree)$.
	
	\begin{definition}
		Let $\DerTree$ be a derivation tree such that $\res(\DerTree) = \Pi \vdash p$ for $p \in \PrFm$. Given any sequent $\Gamma, \Delta \vdash C$ not equal to $\vdash p$, let us define the function $\Aug(\DerTree;(\Gamma;\Delta;C))$ that transforms $\DerTree$ into an incomplete derivation tree $\Tree$ such that $\res(\Tree) = \Gamma, \Pi, \Delta \vdash C$. The definition is by induction on the size of $\DerTree$.
		\begin{enumerate}
			\item $\Aug(p \vdash p;(\Gamma;\Delta;C)) \eqdef \Gamma, p, \Delta \vdash C$.
			\item Assume that $\DerTree$ is of the following form:
			$$
			\infer[(r)]
			{
				\Pi \vdash p
			}
			{
				(\DerTree_n)_{n < \alpha}
			} 
			$$
			Here $\alpha \le \omega$ and $(r)$ is one of the rules of $\ACTMult$ but not $(\bs L)$; $\DerTree_n$ are derivation subtrees with the conclusions being premises of $(r)$. Note that $\res(\DerTree_n)$ is of the form $\Pi_n \vdash p$ because, for any left rule except for $(\bs L)$, the succedent of each premise coincides with the succedent of the conclusion. Then
			$$
			\Aug(\DerTree;(\Gamma;\Delta;C)) \eqdef 
			\infer[(r)]
			{
				\Gamma, \Pi, \Delta \vdash C
			}
			{
				(\Aug(\DerTree_n;(\Gamma;\Delta;C)))_{n < \alpha}
			}
			$$
			
			\item Assume that $\DerTree$ is of the following form:
			$$
			\infer[(\bs L)]
			{
				\Pi \vdash p
			}
			{
				\DerTree_1
				&
				\DerTree_2
			} 
			$$
			Here $\res(\DerTree_1) = \Xi, A, \Psi \vdash p$, $\res(\DerTree_2) = \Theta \vdash B$, and $\Pi = \Xi, \Theta, B \bs A, \Psi$. Then
			$$
			\Aug(\DerTree;(\Gamma;\Delta;C)) \eqdef  
			\infer[(r)]
			{
				\Gamma, \Pi, \Delta \vdash C
			}
			{
				\Aug(\DerTree_1;(\Gamma;\Delta;C))
				&
				\DerTree_2
			}
			$$
			
		\end{enumerate}
		
	\end{definition}
	
	\begin{example}
		$$
		\DerTree = 
		\infer[(\bs L)]
		{
			p, p \bs q \vdash q
		}
		{
			p \vdash p & q \vdash q
		}
		\quad\Rightarrow\quad
		\Aug(\DerTree;(X,Y;;Z))
		=
		\infer[(\bs L)]
		{
			X,Y, p, p \bs q \vdash Z
		}
		{
			X, Y, p \vdash Z & q \vdash q
		}
		$$
		
	\end{example}
	The following lemma can easily be proved by transfinite induction on $\rho(\Pi \vdash p)$.
	
	\begin{lemma}
		Let $\DerTree$ be a derivation tree of $\Pi \vdash p$. Then $\Tree = \Aug(\DerTree;(\Gamma;\Delta;C))$ is an incomplete derivation tree such that $\res(\Tree) = \Gamma, \Pi, \Delta \vdash C$ and $\hyp(\Tree)$ consists of only one sequent $\Gamma, p, \Delta \vdash C$.
	\end{lemma}
	
	\begin{definition}
		Let $\Tree$ be an incomplete derivation tree such that $\hyp(\Tree) = \{\Theta \vdash C\}$; let $\DerTree$ be a derivation tree such that $\res(\DerTree) = \Theta \vdash C$. Then $\Tree[\DerTree]$ is the tree obtained from $\Tree$ by replacing each its leaf with the label $\Theta \vdash C$ by $\DerTree$.
	\end{definition}
	
	Clearly, $\Tree[\DerTree]$ is a derivation tree.
	
	\begin{definition}\label{def_DAB}
		Let $\DerTree$ be a derivation tree of $\Gamma, A, \Delta \vdash D$. Let $B$ be some formula. Then $\DerTree\vert_{A \wedge B}(\Gamma;\Delta)$ is a tree defined as follows:
		\begin{enumerate}
			\item If $A$ is principal in the last rule application in $\DerTree$, then 
			$$
			\DerTree\vert_{A \wedge B}(\Gamma;\Delta)
			\eqdef
			\infer[(\wedge L_1)]{\Gamma, A \wedge B, \Delta \vdash D}{\DerTree}
			$$
			\item Let $A$ be not principal and let the last rule application look as follows:
			$$
			\infer[]{\Gamma, A, \Delta \vdash D}{(\DerTree_n)_{n < \alpha}}
			$$
			
			Here $\res(\DerTree_n) = \Theta_n \vdash D_n$. For some of $n < \alpha$, it holds that $\Theta_n = \Gamma_n, A, \Delta_n$ where the distinguished occurrence of $A$ corresponds to the distinguished occurrence of $A$ in $\Gamma, A, \Delta \vdash D$. For such $n$, let $\DerTree^\prime_n = \DerTree_n\vert_{A \wedge B}(\Gamma_n;\Delta_n)$; otherwise, let $\DerTree^\prime_n = \DerTree_n$. Then
			$$
			\DerTree\vert_{A \wedge B}(\Gamma;\Delta)
			\eqdef
			\infer[]{\Gamma, A \wedge B, \Delta \vdash D}{(\DerTree^\prime_n)_{n < \alpha}}
			$$
		\end{enumerate}
	\end{definition}
	The construction presented in Definition \ref{def_DAB} is developed to construct a derivation of $\Gamma,A \wedge B, \Delta \vdash D$ such that $(\wedge L_1)$ is applied to $A \wedge B$ only where either $A$ or $B$ must be principal. Clearly, $\DerTree\vert_{A \wedge B}(\Gamma;\Delta)$ is a derivation tree of the sequent $\Gamma,A \wedge B,\Delta \vdash D$.
	
	Using the constructions defined above we can prove Lemma \ref{lemma_basic_derivation}.
	
	\begin{proof}[Proof of Lemma \ref{lemma_basic_derivation}]
		The proof is by transfinite induction on $\rho(\Theta \vdash D)$. The axiom case is trivial. To prove the induction step, fix a derivation of $\Theta \vdash D$ and consider the last rule applied in it.
		
		\textbf{Case 1.} The last rule applied is $(\bs L)$ and the principal formula is $p \bs A$ for $p \in \PrFm$:
		$$
		\infer[(\bs L)]
		{
			\Gamma, \Pi, p \bs A, \Delta \vdash D
		}
		{
			\DerTree_1
			&
			\DerTree_2
		}
		$$
		
		Here $\res(\DerTree_1) = \Gamma, A, \Delta \vdash D$, $\res(\DerTree_2) = \Pi \vdash p$. Let us apply the induction hypothesis to both these sequents; consequently, there are basic derivations $\DerTree^\prime_1$, $\DerTree^\prime_2$ such that $\res(\DerTree^\prime_i) = \res(\DerTree_i)$. Consider the incomplete derivation tree $\Tree = \Aug(\DerTree^\prime_2;(\Gamma; p \bs A, \Delta; D))$. Its root is $\Gamma, \Pi, p \bs A, \Delta \vdash D = \Theta \vdash D$. The only hypothesis in $\Tree$ is $\Gamma, p, p \bs A, \Delta \vdash D$.
		
		Let $\DerTree^{\prime\prime}_1$ be the following derivation tree:
		$$
		\infer[(\bs L)]
		{
			\Gamma, p, p \bs A, \Delta \vdash D
		}
		{
			\DerTree^\prime_1
			&
			p \vdash p
		}
		$$
		
		Then $\Tree[\DerTree^{\prime\prime}_1]$ is a derivation of $\Theta \vdash D$. One can easily verify that it is basic. 
		
		\textbf{Case 2.} The last rule is $(\cdot R)$ and the principal formula $D$ equals $p_1 \cdot p_2$ for $p_1, p_2 \in \PrFm$:
		$$
		\infer[(\cdot R)]
		{
			\Theta_1, \Theta_2 \vdash p_1 \cdot p_2
		}
		{
			\DerTree_1
			&
			\DerTree_2
		}
		$$
		
		Here $\res(\DerTree_i) = \Theta_i \vdash p_i$ (for $i = 1,2$). Let us apply the induction hypothesis to both these sequents; consequently, there are basic derivations $\DerTree^\prime_1$, $\DerTree^\prime_2$ such that $\res(\DerTree^\prime_i) = \res(\DerTree_i)$. Let $\Tree_1 = \Aug(\DerTree_1;(;\Theta_2;D))$ (the semicolon right after the left bracket means that the first sequence of formulas is empty). This is an incomplete derivation tree such that $\hyp(\Tree_1) = \{p_1, \Theta_2 \vdash D\}$. Now, let $\Tree_2 = \Aug(\DerTree_2;(p_1;;D))$; clearly, $\res(\Tree_2) = p_1, \Theta_2 \vdash D$ and $\hyp(\Tree_2) = \{p_1,p_2 \vdash p_1 \cdot p_2\}$. Finally, let $\DerTree_3$ be the only derivation tree of $p_1,p_2 \vdash p_1 \cdot p_2$. Then $\Tree_1[\Tree_2[\DerTree_3]]$ is the desired derivation tree.
		
		\textbf{Case 3.} Let the last rule applied be $(\wedge L_i)$ and let the principal formula be $A_1 \wedge A_2$ where $A_i = B \bs A$. Assume that $i=1$ (the case $i=2$ is dealt with similarly), i.e. $A_1 \wedge A_2 = (B \bs A) \wedge A_2$. The derivation of $\Theta \vdash D$ thus looks as follows:
		$$
		\infer[(\wedge L_1)]
		{
			\Gamma, (B \bs A) \wedge A_2, \Delta \vdash D
		}
		{
			\DerTree_0
		}
		$$
		
		Here $\res(\DerTree_0) = \Gamma, B \bs A, \Delta \vdash D$. Let us apply the induction hypothesis to the latter sequent and obtain a basic derivation $\DerTree^\prime_0$ for it. Finally, the following derivation is basic:
		$$
		\infer[(\wedge L_1)]
		{
			\Gamma, (B \bs A) \wedge A_2, \Delta \vdash D
		}
		{
			\DerTree^\prime_0\vert_{(B \bs A) \wedge A_2}(\Gamma;\Delta)
		}
		$$
		
		\textbf{Case 4.} None of the cases considered above is the case. In general, the derivation of $\Theta \vdash D$ looks as follows:
		$$
		\infer[(r)]
		{
			\Theta \vdash D
		}
		{
			(\DerTree_i)_{i < \alpha}
		}
		$$
		
		Here $\alpha \le \omega$. One can apply the induction hypothesis to $\res(\DerTree_i) = (\Theta_i \vdash D_i)$ for each $i<\alpha$ and obtain a basic derivation $\DerTree^\prime_i$ of $\Theta_i \vdash D_i$. Then the following is a basic derivation as well:
		$$
		\infer[(r)]
		{
			\Theta \vdash D
		}
		{
			(\DerTree^\prime_i)_{i < \alpha}
		}
		$$
		
		This completes the proof.
	\end{proof}
	
	\subsection{Hyperarithmetical Hierarchy}\label{ssec_hyperarithmetic}
	
	The hyperarithmetical hierarchy extends the notions of $\Delta^0_\alpha$, $\Sigma^0_\alpha$ and $\Pi^0_\alpha$-sets to all computable ordinals $\alpha < \omega_1^{CK}$. The standard definitions use Kleene's $\mathcal{O}$ notation for ordinals. It is, however, quite complex; for example, checking whether a number $a \in \Nat$ is a notation of some ordinal is $\Pi^1_1$-complete; besides, each infinite ordinal has infinitely many different notations. Since we shall consider the hyperarithmetical hierarchy only up to the level $\omega^\omega$ we can use a simple bijective notation. Namely, we can exploit the fact that each ordinal below $\omega^\omega$ is of the form $\omega^{N} \cdot g_N +\dotsc+\omega^{0} \cdot g_0$ where $g_i \in \Nat$. Thus one can encode this ordinal by the tuple $\langle g_0,\dotsc,g_N \rangle$. Recall that there is an effective G\"{o}del numbering of tuples of natural numbers, e.g. the standard one: $\langle i_1,\dotsc,i_n \rangle \leftrightarrow p_1^{i_1}\cdot\dotsc\cdot p_n^{i_n}$ where $p_j$ is the $j$-th prime number.  Hereinafter, we usually equate a tuple from $\Nat^{<\Nat}$ with its encoding.
	
	\begin{definition}
		For $\alpha < \omega^\omega$, let $\pi(\alpha) = \langle g_0,\dotsc,g_N \rangle$ such that $\alpha = \omega^{N} \cdot g_N +\dotsc+\omega^{0} \cdot g_0$ where $g_i \in \Nat$ and $g_N \ne 0$. The tuple $\pi(\alpha)$ is called the \emph{polynomial notation for $\alpha$}. 
	\end{definition}
	
	Consider the set $\PN = \{\langle g_0,\dotsc,g_N \rangle \mid n \ge 0, g_i \in \Nat, g_N \ne 0 \} = \Nat \setminus \{0\}$ of all polynomial notations. Each notation corresponds to exactly one ordinal below $\omega^\omega$. Moreover, let us define the well-ordering $\prec$ on $\PN$ induced by the ordering $\in$ on ordinals: $\langle g_0,\dotsc,g_M \rangle \prec \langle h_1,\dotsc,h_N \rangle$ if and only if either $M < N$ or $M=N$ and $ \langle g_0,\dotsc,g_M \rangle$ is less than $ \langle h_1,\dotsc,h_N \rangle$ according to the antilexicographic order. Thus, $(\PN,\prec)$ is a computable well-ordering of the order type $\omega^\omega$. Using it we can define the jump hierarchy:
	\begin{align*}
		H(0) &\eqdef \emptyset;
		\\
		H(\alpha+1) &\eqdef H(\alpha)^\prime \mbox{ for $\alpha < \omega^\omega$;}
		\\
		H(\alpha) &\eqdef \left\{\langle \pi(\beta),n \rangle \mid \beta < \alpha, \; n \in H(\beta)\right\} \mbox{ for $\alpha \le \omega^\omega$ being limit.}
	\end{align*}
	
	For $\alpha \le \omega^\omega$, the set $H(\alpha)$ defined in this way is Turing-equivalent to the set $H(a)$ defined using Kleene's $\mathcal{O}$ for some $a \in \mathcal{O}$ such that $\vert a \vert_{\mathcal{O}} = \alpha$ \cite[Exercise 16-80]{Rogers67}. Thus our definition is equivalent to the standard one. Note that $\Delta^0_{\omega^\omega}$ consists of sets computable relative to $H(\omega^\omega)$, hence $H(\omega^\omega)$ is by the definition $\Delta^0_{\omega^\omega}$-complete in terms of Turing reducibility.
	
	Another way of defining the hyperarithmetical hierarchy is by means of computable infinitary formulas. Let us give the definitions according to \cite{Montalban23}. Firstly, let us recall some basic notation. Let $f : \Nat \to \Nat$ be some function. Then $\Phi^f_e:\Nat \to \Nat$ denotes the $e$-th Turing operator when it uses $f$ as oracle. Let $W^f_e$ denote the domain of $\Phi^f_e$. If $f(x) \equiv 0$, then let us denote these functions and sets as $\Phi_e$ and $W_e$ resp. For $\sigma \in \Nat^{< \Nat}$, $\Phi^\sigma_{e}(n)$ equals $\Phi^f_e(n)$ if $\sigma = (f(0),\dotsc,f(k))$ is a finite prefix of $f$ and the $e$-th Turing operator $\Phi^f_e$, given $n$ as input, runs for at most $k$ steps, uses only entries from $\sigma$ when appealing to the oracle and finishes in the final state (i.e. completes the computation process); otherwise, $\Phi^\sigma_{e}(n)$ diverges. Let $\HALT(n,e,\sigma)$ be the predicate ``$\Phi^\sigma_{e}(n)$ converges''. Let $\Halt(n,e,y) \eqdef \HALT(n,e,0^y)$. Note that $\HALT$ and $\Halt$ are computable.
	
	\begin{remark}\label{remark_Halt}
		$n \in W_e$ if and only if $\exists y \Halt(n,e,y)$ is true.
	\end{remark}
	
	\begin{definition}[\cite{Montalban23}]\label{def_computable_infinitary_formulas}
		We fix a countable set of variables $\{x_n \mid 1 \le n < \omega\}$. $\Sigma_0$- and $\Pi_0$-formulas are finitary quantifier-free formulas in the language $(+,\times,0,1,<)$ of arithmetic. For each $i \in \Nat$, let us choose an effective numbering of finitary quantifier-free arithmetical formulas such that the set of their free variables is contained in $\{x_1,\dotsc,x_i\}$; in other words, we fix a computable bijection from the set of such formulas onto $\Nat$. By $S^\Sigma_{0}$ we denote the set of quadruples $\langle \Sigma, \pi(0), i, c \rangle$ where $c$ is the number of a quantifier-free formula $\varphi$ with free variables from $\{x_1,\dotsc,x_i\}$ according to the fixed numbering. Let $\varphi^{\Sigma_0}_{c,i}(x_1,\dotsc,x_i) = \varphi^{\Pi_0}_{c,i}(x_1,\dotsc,x_i)$ denote this formula. The set $S^\Pi_0$ consists of quadruples of the form $\langle \Pi, \pi(0), i, c \rangle$ where $c$ and $i$ are defined in the same way\footnote{Formally, in these tuples, $\Sigma$ and $\Pi$ are just two different numbers, e.g. $1$ and $2$.}.
		\\
		For $0 < \alpha < \omega^\omega$, let $S^\Sigma_{\alpha}$ be the set of quadruples $\langle \Sigma,\pi(\alpha),j,e \rangle$ where $\alpha < \omega^\omega$ and $j,e \in \omega$. The set $S^\Pi_{\alpha}$ consists of similar quadruples of the form $\langle \Pi,\pi(\alpha),j,e \rangle$.
		Quadruples $\langle \Sigma,\pi(\alpha),j,e \rangle$ and $\langle \Pi,\pi(\alpha),j,e \rangle$ (we call them \emph{indices}) correspond to the infinitary formulas $\varphi^{\Sigma_\alpha}_{e,j}$ and $\varphi^{\Pi_\alpha}_{e,j}$ defined below by mutual transfinite recursion:
		\[
		\varphi^{\Sigma_\alpha}_{e,j}(x_1,\dotsc,x_j) = \bigdoublevee\limits_{\substack{\beta < \alpha \\ \langle \pi(\beta),k,e^\prime\rangle  \in W_e}}
		\exists x_{j+1} \dotsc \exists x_{j+k} \; \varphi^{\Pi_\beta}_{e^\prime,j+k}(x_1,\dotsc,x_{j+k}).
		\]
		\[
		\varphi^{\Pi_\alpha}_{e,j}(x_1,\dotsc,x_j) = \bigdoublewedge\limits_{\substack{\beta < \alpha \\ \langle \pi(\beta),k,e^\prime \rangle  \in W_e}}
		\forall x_{j+1} \dotsc \forall x_{j+k} \; \varphi^{\Sigma_\beta}_{e^\prime,j+k}(x_1,\dotsc,x_{j+k}).
		\]
	\end{definition}
	
	The definition of truth of a formula is standard, it is inherited from the infinitary logic $\LogicOOO$. If $\varphi^{\Sigma_\alpha}_{e,j}(x_1,\dotsc,x_j)$ is true under an assignment of variables $(n_1,\dotsc,n_j)$, we write $\Nat \vDash \varphi^{\Sigma_\alpha}_{e,j}(n_1,\dotsc,n_j)$ (the same is for $\Pi_\alpha$).
	
	\begin{remark}\label{remark_manipulating_formulas}
		Although we have introduced computable infinitary formulas only in the ``canonical form'' we can apply usual logical operations to them \cite[Theorem 7.1]{AshK00}. E.g. given indices $c_1$, $c_2$ for two computable $\Sigma_\alpha$-formulas $\varphi_1$ and $\varphi_2$, one can compute the index of a $\Sigma_\alpha$-formula logically equivalent to $\varphi_1 \vee \varphi_2$ or to $\varphi_1 \wedge \varphi_2$; also, given an index $c$ for a $\Sigma_\alpha$-formula $\varphi$, one can compute the index of a $\Pi_\alpha$-formula logically equivalent to $\neg \varphi$. Besides, there is a computable function $\mathrm{sub}$ such that 
		$\mathrm{sub}(\langle \langle X,\pi(\alpha),i,e \rangle ,  \langle n_1,\dotsc,n_i \rangle \rangle )$
		is the index of the formula $\varphi^{X_\alpha}_{e,i}\left(n_1,\dotsc,n_i\right)$.
		To prove these statements one needs effective transfinite recursion \cite[I.4.1]{Montalban23}. 
	\end{remark}
	
	Effective transfinite recursion is a very useful tool that allows one to manipulate infinitary computable formulas in a computable way. Let us recall this method according to \cite{Montalban23}. For $p \in \PN$ and $e \in \Nat$, let $e\restriction_{\prec p}$ be an index for the computable function obtained by restricting the domain of $\Phi_e$ to $\PN \restriction_{\prec p}$: $\Phi_{e\restriction_{\prec p}}(p^\prime) = \Phi_e(p^\prime)$ for $p^\prime \prec p$ and $\Phi_{e\restriction_{\prec p}}(p^\prime)$ diverges otherwise. Then, for every computable function $\Psi$, there exists an index $e$ for a computable function $\Phi_e$ such that $\Phi_e(p) = \Psi(p,e \restriction_{\prec p})$. The function $\Psi$ is a recursive description of $\Phi_e$ because, when $\Psi(p,e\restriction_{\prec p})$ is being computed, the function $\Phi_{e\restriction_{\prec p}}$ can be used. We are going to use effective transfinite recursion usually without explicitly defining $\Psi$; the only proof where $\Psi$ will be described in detail is that of Proposition \ref{prop_Der}.
	
	\begin{definition}
		Let us define the satisfaction predicate as follows:
		$$
		\SAT_{<\beta} \eqdef \left\{\langle \langle X,\pi(\alpha),i,e\rangle , \langle n_1,\dotsc,n_i \rangle \rangle  \mid \alpha < \beta, \, X \in \{\Sigma,\Pi\}, \, \Nat \vDash \varphi^{X_\alpha}_{e,i}(n_1,\dotsc,n_i)\right\}.
		$$
	\end{definition}
	We shall be mainly concerned with the satisfaction predicate $\SATOO$ for the level $\omega^\omega$. Its relation to the jump hierarchy is explained by the following proposition.
	
	\begin{proposition}\label{proposition_SAT_H}
		$\SATOO$ is m-equivalent to $H(\omega^\omega)$.
	\end{proposition}
	
	The proof of this proposition is essentially presented in \cite[Lemma V.15, Theorem V.16]{Montalban23}, \cite[Theorem 7.5]{AshK00}. Namely, to show that $H(\omega^\omega)$ is m-reducible to $\SATOO$, effective transfinite recursion is used \cite[I.4.1]{Montalban23} to define the function $f:\PN \to \Nat$ such that $f(\pi(\alpha)) \in S^\Sigma_{\alpha}$ is the index of a $\Sigma_{\alpha}$-formula $\theta^{[\alpha]}(x_1)$ such that $\theta^{[\alpha]}(n)$ is true if and only if $n \in H(\alpha)$. If $\alpha = 0$, then $f(\pi(\alpha))$ is an index of the formula $\theta^{[0]}(x_1) = (x_1+1 = 0)$. If $\alpha = \beta+1$, then we assume that $f(\pi(\beta))$ is defined and we can use it to define $f(\pi(\beta+1))$. Let $\theta^{[\beta]}$ be the $\Sigma_{\beta}$-formula with the index $f(\pi(\beta))$. Then $\theta^{[\alpha]}$ can be effectively defined as follows:
	$$
		\theta^{[\alpha]}(x_1) \eqdef \bigdoublevee\limits_{n \in \Nat} \bigdoublevee\limits_{\substack{\sigma \in 2^{< \Nat} \\ \Halt(n,n,\sigma)}} \left(x_1 = {n}\right) \wedge \bigwedge\limits_{\substack{i < \vert \sigma \vert \\ \sigma(i) = 1}} \theta^{[\beta]}(i) \wedge \bigwedge\limits_{\substack{i < \vert \sigma \vert \\ \sigma(i) = 0}} \neg \theta^{[\beta]}(i)
	$$
	If $\gamma$ is a limit ordinal, then $\theta^{[\gamma]}(x,y) \eqdef \bigdoublevee\limits_{\beta < \gamma} \left(x = {\pi(\beta)}\right) \wedge \theta^{[\beta]}(y)$. 
	
	Conversely, to show that $\SATOO$ is m-reducible to $H(\omega^\omega)$ it suffices to define the function $g:\PN \to \Nat$ such that $\inp =  \langle \langle X,\pi(\alpha),i,e\rangle , \langle n_1,\dotsc,n_i \rangle \rangle \in \SATOO$ if and only if $\Phi_{g(\pi(\alpha))}(\inp) \in H(\omega^\omega)$. Again, this is done by effective transfinite recursion (see details in \cite[Theorem V.16]{Montalban23}).
	
	\begin{corollary}
		$\SATOO$ is $\Delta^0_{\omega^\omega}$-complete under Turing reductions.
	\end{corollary}
	
	In what follows, showing that $\SATOO$ and the derivability problem in $\ACTMult$ are recursively isomorphic would suffice to establish $\Delta^0_{\omega^\omega}$-completeness of $\ACTMult$.

	\subsection{Computable Functions as String Rewriting Systems}\label{ssec_SR}
	
	In the last preliminary subsection, we aim to describe computable functions in terms of string rewriting systems (also known as semi-Thue systems) instead of Turing machines, which is convenient for further encoding. Recall that a string rewriting system $\mathit{SR}$ over an alphabet $\mathcal{A}$ is a finite set of rules of the form $\alpha \to \beta$ where $\alpha$, $\beta$ are arbitrary strings with symbols from $\mathcal{A}$. We write $u\alpha v \Rightarrow_{\mathit{SR}} u \beta v$ if $(\alpha \to \beta) \in \mathit{SR}$.
	
	\begin{definition}
		A Turing machine is a tuple $TM = (Q,\Tape, \Input,\Output, \blank, \Trans, q_0, q_a)$ where  $Q$ is a finite set of \emph{states}; $\Tape$ is the alphabet of tape symbols; $\Input, \Output \subseteq \Tape$ are alphabets of input and output symbols resp.; $\blank \in \Tape \setminus (\Input \cup \Output)$ is the distinguished \emph{blank} symbol; $\Trans \subseteq \left(Q \setminus \{q_a\} \right) \times \Tape \times Q \times \Tape \times \{R,L,N\}$ is the transition relation; $q_0 \in Q$ is the initial state; $q_a \in Q$ is the accepting state.
	\end{definition}
	
	In our setting, a tape of a Turing machine is infinite in both sides. A configuration of $\TM$ can be described by a string $uaqv$ where $u,v \in \Tape^\ast$, $a \in \Tape$ and $q \in Q$. Let us agree on that the head position in this configuration is $a$. An initial configuration is of the form $uq_0$ where $u \in \Input^\ast$. A final configuration is of the form $q_a w$ such that $w \in \Output^\ast$. If $uq_0$ is transformed into $q_a w$, then we say that $\TM$ transforms $u$ into $w$.
	
	\begin{definition}
		Given a Turing machine $\TM = (Q,\Tape, \Input, \Output, \blank, \Trans, q_0, q_a)$ and the symbols $\final,a_L,a_R$ not from $\Tape \cup Q$, let $\SR(\TM;a_L,a_R,\final)$ be the string rewriting system consisting of the following rules \cite{BookO93}:
		
		\begin{multicols}{2}
			\begin{itemize}
				\item $a_R \to q_0 a_R^\prime$;
				\item $a q \to b r, \; (q,a,r,b,N) \in \Trans$;
				\item $a_{L} q \to a_L b r, \; (q,\blank,r,b,N) \in \Trans$;
				\item $a q \to r b, \; (q,a,r,b,L) \in \Trans$;
				\item $a_{L} q \to a_L r b, \; (q,\blank,r,b,L) \in \Trans$;
				\item $a q c \to b c r$, $(q,a,r,b,R) \in \Trans$, $c \in \Tape$;
				\item $a q a_R^\prime \to b \blank r a_R^\prime,  (q,a,r,b,R) \in \Trans$;
				\item $a_L q a_R^\prime \to a_L b \blank r a_R^\prime,  (q,\blank,r,b,R) \in \Trans$;
				\item $\blank q_a \to q_a$; \quad $\blank a_R^\prime \to a_R^\prime$;
				\item $a_L q_a \to a_L a_L^\prime$;
				\item $a_L^\prime c \to c a_L^\prime , \; c \in \Output$; \quad $a_L^\prime a_R^\prime \to \final$.
			\end{itemize}
		\end{multicols}
	\end{definition}
	
	\begin{proposition}[cf. {\cite[Lemma 2.5.2]{BookO93}}]\label{prop_TM_SR}
		$a_L u a_R \Rightarrow_{\SR(\TM;a_L,a_R,\final)}^\ast v \final$ if and only if $v = a_L w$ and $\TM$ transforms $u$ into $w$.
	\end{proposition}
	
	The proof is by straightforwardly inspecting the rules of $\SR(\TM;a_L,a_R,\final)$. We shall use its following corollary.
	
	\begin{corollary}[Church-Turing thesis for string rewriting systems]\label{corollary_Church-Turing}
		Given a computable partial function $f:\Input^\ast \to \Output^\ast$ and the symbols $a_L,a_R,\final \notin \Input \cup \Output$, one can effectively find a string rewriting system $\SR(f;a_L,a_R,\final)$ such that $a_L u a_R \Rightarrow_{\SR(f;a_L,a_R,\final)}^\ast v \final$ if and only if $v = a_L f(u)$.
	\end{corollary}
	
	\begin{definition}
		We say that a string rewriting system $(a_L,a_R,\final)$-implements a computable function $f:\Input^\ast \to \Output^\ast$ if and only if it equals $\SR(f;a_L,a_R,\final)$.
	\end{definition}
	
	\section{Main Results}\label{sec_main}
	
	The main result of this work is Theorem \ref{th_main_complexity} presented in Section \ref{sec_introduction}. It follows from the following theorem.
	
	\begin{theorem}\label{th_main_1-equivalence}
		$\SATOO$ is recursively isomorphic to the derivability problem for $\ACTMult$.
	\end{theorem}
	
	As a consequence, by Proposition \ref{proposition_SAT_H}, we can derive the following. 
	\begin{corollary}\label{corollary_main_m-reducibility}
		The derivability problem for $\ACTMult$ is m-equivalent to $H(\omega^\omega)$.
	\end{corollary}
	
	\subsection{Construction}\label{ssec_construction}
	
	Let us prove Theorem \ref{th_main_1-equivalence}. We start with constructing a one-one reduction of $\SATOO$ to the derivability problem for $\ACTMult$. Given $\inp = \langle \langle X,\pi(\alpha),i,e \rangle ,\langle n_1,\dotsc,n_i \rangle \rangle $ as an input, we shall define the sequent $a_L, a_1^{\inp}, a_X, \energy, \Energy_{h_M}, \dotsc, \Energy_{h_1} \vdash a_L \cdot \okay$ which is derivable if and only if $\inp \in \SATOO$. Here $\Energy_k$ are formulas and $h_1,\dotsc,h_M \in \Nat$ are non-negative integers such that $\omega^{h_1} + \dotsc + \omega^{h_M} = \alpha+1$; $a_L,a_1,a_\Sigma,a_\Pi,\energy,\okay \in \PrFm$ are primitive formulas. Recall that $a^n$ is $a,\dotsc,a$ repeated $n$ times. The formula $\Energy_0$ will be designed to simulate the algorithm that does one step of checking whether a given computable infinitary formula is true under a given assignment of variables. For example, if the formula is a disjunction of formulas, $\Energy_0$ ``chooses'' one of the disjuncts and forms a ``subroutine'' asking whether this disjunct is true under the assignment of formulas. The formula $\Energy_n$ will be defined using nested subexponentials, and it will contain $\Energy_0$ inside. Roughly speaking, $\Energy_n$'s role is to make sufficiently many copies of $\Energy_0$. 
	
	The construction presented below is quite large, and it is hard to explain the meaning of each its part right after its definition. Probably, it would be helpful for the reader to firstly take a look at the definitions below without inspecting them too carefully and then proceed to Sections \ref{ssec_BTA} and \ref{ssec_proof}. While reading the proofs presented in those sections, the reader would return to the below definitions and grasp their details.
	
	The following expressions will be frequently used in the main construction.
	
	\begin{definition}
		\leavevmode
		\begin{multicols}{3}
			\begin{enumerate}
				\item $[ B ]^? A \eqdef B \bs (B \cdot A)$;
				\item $[ B ]^\rightarrow A \eqdef B \bs (A \cdot B)$;
				\item $\OKAY \eqdef \okay \bs \okay$.
			\end{enumerate}
		\end{multicols}
	\end{definition}
	We start with encoding string rewriting systems using formulas of $\ACTMult$. We assume that the symbols of the alphabet used by a string rewriting system are primitive formulas. Additionally, we need new primitive formulas $\wait,\go,\fail,\final$.
	
	\begin{definition}\label{def_fm}
		Let $r = (c_1\dotsc c_m \to b_1 \dotsc b_n)$ be a rule of a string rewriting system (we assume that $m,n > 0$). Then $\fm(r) \eqdef c_m \bs \dotsc \bs c_1 \bs \left(b_1 \cdot \dotsc \cdot b_n \cdot \nabla \wait\right)$.
	\end{definition}
	The formula $\fm(r)$ is designed to model one application of the rule $r$ of a string rewriting sytem.
	
	\begin{example}
		Let $r = (ac \to bba)$. Then $\fm(r) = c \bs a \bs (b \cdot b \cdot a \cdot \nabla \wait)$. Note that, if the sequent $\Gamma, b, b, a, \nabla \wait, \Delta \vdash C$ is derivable, then so is $\Gamma, a,c, \Delta \vdash C$. The role of $\nabla \wait$ will become clear later; we need this formula to control the derivation better.
	\end{example}
	
	\begin{definition}\label{def_Rule}
		Let $SR$ be a string rewriting system. Then 
		$$
		\Rule_{SR} \eqdef \go \big\bs \bigwedge\limits_{r \in SR} (\nabla \fm(r) \cdot \wait \bs \go).
		$$
	\end{definition}
	
	Roughly speaking, the formula $\Rule_{SR}$ is the collection of all formulas $\fm(r)$ for $r \in SR$; i.e. one can choose any of the rules of $SR$ and to apply it. Again, the role of $\go$ and $\wait$ will become clear later.
	
	\begin{definition}\label{def_Technical}
		Given a string rewriting system $SR$ over $\mathcal{A}$ and a function $f:\{a_1,a_2\} \to \Fm$ where $\{a_1,a_2\} \subseteq \mathcal{A}$ (hereinafter, $a_1$ and $a_2$ are two fixed primitive formulas), let
		$$
		\Technical(SR,f) \eqdef 
		\go \bs \left( a_R \cdot
		\go \cdot \bang\Rule_{SR} \cdot \go \bs \final \bs \left( a_1 \bs f(a_1) \wedge a_2 \bs f(a_2) \right) \right).
		$$
	\end{definition}
	
	Inside the formula $\Technical(SR,f)$, there is the formula $\bang \Rule_{SR}$, which is able to make arbitrarily many copies of $\Rule_{SR}$. Informally, $\Technical(SR,f)$'s role is to simulate a derivation in the string rewriting system $SR$.
	
	\begin{definition}\label{def_Energy}
	Let us define the formulas $\Energy_{\Sigma}$, $\Energy_{\Pi}$, $\Energy_0$, and $\Energy_k$ for $k \in \Nat$, which play the main role in the construction. Below, $a_R, a_2$ are primitive formulas.
	
	\begin{enumerate}
		\item $\Energy_{\Sigma} \eqdef 
		\go \cdot \Technical\left(\sr_0,f_0\right) \cdot 
		\left(\OKAY \wedge [ \go ]^?\bang [ \go ]^\rightarrow a_2 \right) \cdot
		\left(\OKAY \wedge \Technical\left(\sr_1,f_\Sigma\right) \right) $.
		\begin{enumerate}
			\item $\sr_0$ is the string rewriting system that $(a_L,a_R,\final)$-implements the following computable function $\mathcal{F}_0$ (cf. Corollary \ref{corollary_Church-Turing}). Its input alphabet is $\{a_1\}$.
			If the input string is of the form $a_1^\inp$ where $\inp = \langle c,\langle n_1,\dotsc,n_i \rangle \rangle$, then the algorithm firstly checks whether $c \in S^X_0$. If $c = \langle X, \pi(0), i, e \rangle \in S^X_0$ for $X \in \{\Sigma,\Pi\}$, then the algorithm checks whether $\inp \in \SATOO$. This is decidable because $\varphi^{\Sigma_0}_{e,i} = \varphi^{\Pi_0}_{e,i}$ is a finitary quantifier-free formula. If $\inp \notin \SATOO$, then the function diverges (i.e. the algorithm halts in some non-accepting state). If $\inp \in \SATOO$, then the function returns $a_1$. If $c = \langle X, p, i, e \rangle$ for $p \succ \pi(0)$, then it returns the string $a_1^\inp a_2$. Otherwise, it diverges. Recall that $f_0(u)=v$ iff $a_L u a_R \Rightarrow^\ast_{\sr_0} a_L v \final$.
			
			\item $f_0:\{a_1,a_2\} \to \Fm$ is defined as follows: $f_0(a_1) = \okay$; $f_0(a_2) = \go$. 
			
			\item $\sr_1$ is the string rewriting system that $(a_L,a_R,\final)$-implements the following computable function $\mathcal{F}_1$ in the sense of Corollary \ref{corollary_Church-Turing}. If the input $w$ is of the form $w = a_1^{\inp_1} a_2^{\inp_2}$ such that 
			\begin{itemize}
				\item $\inp_1 = \langle \langle X, p, i, e \rangle, \langle n_{1}, \dotsc,n_{i}\rangle \rangle$ for $\pi(0) \prec p$ and 
				\item $\inp_2 = \langle c^\prime, y, \langle n_{i+1}, \dotsc,n_{i+j}\rangle \rangle$ where $c^\prime = \langle p^\prime,j,e^\prime \rangle$ for $p^\prime \prec p$ such that $\Halt(c^\prime,e,y)$ holds, 
			\end{itemize}
			then $\mathcal{F}_1(w) = a_1^{\langle \langle Y, p^\prime, i+j, e^\prime \rangle, \langle n_1,\dotsc,n_{i+j} \rangle \rangle} a_2$ where $Y = \Pi$ if $X = \Sigma$ and $Y = \Sigma$ if $X = \Pi$.	Checking these conditions is a computable problem. If the input is not of the form described above, then $\mathcal{F}_1(w) = a_1$.
			
			\item $f_\Sigma:\{a_1,a_2\} \to \Fm$ is defined as follows: $f_\Sigma(a_1) = \fail$; $f_\Sigma(a_2) = a_\Pi \cdot \energy$. 
		\end{enumerate}
		
		\item $\Energy_{\Pi} \eqdef
		\go \cdot \Technical(\sr_0,f_0) \cdot 
		\left(\OKAY \wedge [ \go ]^\rightarrow a_2^\ast\right) \cdot
		\left(\OKAY \wedge \Technical(\sr_1,f_\Pi)\right)$.
		\\
		Here 
		$f_\Pi:\{a_1,a_2\} \to \Fm$ is defined as follows: $f_\Pi(a_1) = \okay$; $f_\Pi(a_2) = a_\Sigma \cdot \energy$. 
		
		\item $\Energy \eqdef (a_\Sigma \bs \Energy_{\Sigma}) \wedge (a_\Pi \bs \Energy_{\Pi})$; \qquad $\Energy_0 \eqdef \OKAY \wedge \energy \bs \Energy$; \qquad $\Energy_{k+1} \eqdef \OKAY \wedge [ \energy ]^? \bang \Energy_k $.
	\end{enumerate}
	\end{definition}
	
	These formulas are used in the main sequent introduced in the following lemma.
	
	\begin{lemma}[main]\label{lemma_main}
		Let $M \ge 0$, $\omega > h_1 \ge \dotsc \ge h_M$, and $\inp \in \Nat$. Consider the sequent
		\begin{equation}\label{eq_seq_main}
			a_L, a_1^{\inp}, a_X, \energy, \Energy_{h_M}, \dotsc, \Energy_{h_1} \vdash a_L \cdot \okay.
		\end{equation}
		Let $\inp = \langle \langle X,\pi(\alpha),i,e \rangle,\langle n_1,\dotsc,n_i \rangle\rangle$.
		\begin{enumerate}
			\item If (\ref{eq_seq_main}) is derivable, then $\inp \in \SATOO$.
			\item If $\inp \in \SATOO$ and $\omega^{h_1}+\dotsc+\omega^{h_M} > \alpha$, then (\ref{eq_seq_main}) is derivable.
		\end{enumerate}
	\end{lemma}
	
	To prove it we firstly develop a method of analysing proofs in $\ACTMult$; this is done in the next subsection.
	
	\subsection{Bottom-Top Analysis}\label{ssec_BTA}
	
	Our aim is to analyse derivability of (\ref{eq_seq_main}). This sequent has the following structure: there is a sequence of primitive formulas followed by a sequence of formulas of the form $\Energy_k$ in the antecedent; the succedent is $a_L \cdot \okay$. Below we develop a technique of analysing sequents of the form $\Gamma, A, \Psi \vdash a_L \cdot \okay$ where $\Gamma$ is a sequence of primitive formulas and $\Psi$ is a sequence of formulas of the form $p \bs B$ or $(p \bs B) \wedge (q \bs C)$ for $p,q \in \PrFm$. Namely, we claim that, if $A$ is not primitive, then one can assume without loss of generality that it is principal. The idea behind the proof is to consider a basic derivation of this sequent and to show that formulas in $\Psi$ cannot be principal as well as $a_L \cdot \okay$. Indeed, a formula of the form $p \bs A$ is principal in a basic derivation only if $p$ stands right to the left from it; similarly, $(p \bs B) \wedge (q \bs C)$ is principal only if $p$ or $q$ is to the left from it. 
	
	\begin{definition}
		A formula is \emph{locked} if it is of the form $p \bs A$ or $(p \bs B) \wedge (q \bs C)$ for some $p,q \in \PrFm$ and $A,B \in \Fm$. A sequence of formulas is \emph{locked} if each formula in it is locked.
	\end{definition}
	
	\begin{lemma}[bottom-top analysis]\label{lemma_bottom-top}
		\leavevmode
		
		\begin{enumerate}
			\item Consider a sequent of the form
			\begin{equation}\label{eq_seq_bottom-top}
				\Gamma, A, \Psi \vdash a_L \cdot \okay
			\end{equation}
			where $\Gamma \in \PrFm^\ast$, $A \in \Fm$ and $\Psi$ is a locked sequence of formulas. Then:
			\begin{enumerate}
				\item\label{item_bottom-top_bs} If $A = p \bs B$ for $p \in \PrFm$, then (\ref{eq_seq_bottom-top}) is derivable iff so is $\Gamma^\prime, B, \Psi \vdash a_L \cdot \okay$ for $\Gamma = \Gamma^\prime, p$; 
				\item\label{item_bottom-top_cdot} If $A = A_1 \cdot A_2$, then (\ref{eq_seq_bottom-top}) is derivable iff so is $\Gamma, A_1,A_2, \Psi \vdash a_L \cdot \okay$;
				\item\label{item_bottom-top_wedge} If $A = A_1 \wedge A_2$, then (\ref{eq_seq_bottom-top}) is derivable iff so is $\Gamma, A_i, \Psi \vdash a_L \cdot \okay$ for some $i\in \{1,2\}$;
				\item\label{item_bottom-top_bang} If $A = \bang B$, then (\ref{eq_seq_bottom-top}) is derivable iff so is $\Gamma, B^n, \Psi \vdash a_L \cdot \okay$ for some $n \in \Nat$;
				\item\label{item_bottom-top_ast} If $A = B^\ast$, then (\ref{eq_seq_bottom-top}) is derivable iff so is $\Gamma, B^n, \Psi \vdash a_L \cdot \okay$ for each $n \in \Nat$.
			\end{enumerate}
			\item Consider the sequent 
			\begin{equation}\label{eq_seq_bottom-top_2}
				\Theta_1, A, \Theta_2 \vdash a_L \cdot \okay
			\end{equation}
			where $\Theta_1,\Theta_2 = \Gamma,\Psi$
			such that $\Gamma \in \PrFm^\ast$ and $\Psi = b \bs E, \Psi^\prime$ is a locked sequence of formulas. Besides, let $b$ be not contained in $\Gamma$. 
			\begin{enumerate}
				\item\label{item_bottom-top_nabla} If $A = \nabla B$, then (\ref{eq_seq_bottom-top_2}) is derivable iff so is $\Xi_1,B,\Xi_2 \vdash a_L \cdot \okay$ for some $\Xi_1$ and $\Xi_2$ such that $\Xi_1,\Xi_2 = \Theta_1,\Theta_2$.
				\item\label{item_bottom-top_bs-2} If $A = p_m \bs \dotsc \bs p_1 \bs B$ for $\{p_1,\dotsc,p_m\} \subseteq \PrFm$, then (\ref{eq_seq_bottom-top_2}) is derivable iff so is $\Theta_1^\prime, B, \Theta_2 \vdash a_L \cdot \okay$ where $\Theta_1 = \Theta_1^\prime, p_1,\dotsc,p_m$. 
			\end{enumerate}
		\end{enumerate}
	\end{lemma}
	
	\begin{proof}
		The ``if'' directions of all iffs hold because, for each statement, the sequent of interest, i.e. (\ref{eq_seq_bottom-top}) or (\ref{eq_seq_bottom-top_2}), can be obtained from the sequent considered in the statement by applying rules of $\ACTMult$. E.g. in \ref{item_bottom-top_nabla}, $\Theta_1, \nabla B, \Theta_2 \vdash a_L \cdot \okay$ is obtained from $\Xi_1,B,\Xi_2 \vdash a_L \cdot \okay$ by the application of $(\nabla L)$ followed by $(\nabla P_i)$ that moves $\nabla B$ to the position between $\Theta_1$ and $\Theta_2$.
		
		Let us prove the ``only if'' directions. Statements \ref{item_bottom-top_cdot} and \ref{item_bottom-top_ast} follow from Corollary \ref{corollary_reversible_rules}. Let us deal with the remaining ones. Assume that (\ref{eq_seq_bottom-top}) is derivable and fix some its basic derivation. Let us consider the last rule application in it and the corresponding principal formula occurrence. It cannot be the succedent $a_L \cdot \okay$ because, by requirement \ref{def_basic_cdot_R} of Definition \ref{def_basic}, this would imply that $\Gamma,A,\Psi = a_L,\okay$ while $A$ is not primitive. It cannot be a formula from $\Psi$ as well. Indeed, assume that a formula $r \bs C$ from $\Psi$ is principal. The requirement \ref{def_basic_bs_L} of Definition \ref{def_basic} implies that the formula $r$ stands right to the left from $r \bs C$. However, each formula in $\Psi$ is preceded by a formula which is not primitive. Similarly, if $(r \bs C) \wedge (s \bs D)$ from $\Psi$ is principal, then the requirements \ref{def_basic_bs_L} and \ref{def_basic_wedge_L} of Definition \ref{def_basic} imply that either $r$ or $s$ is right to the left from the principal formula, which is again not the case. Thus we come up with a contradiction.
		
		In what follows, $A$ must be principal in the basic derivation. Then the ``only if'' parts of statements \ref{item_bottom-top_wedge} and \ref{item_bottom-top_bang} follow from the definitions of rules $(\wedge L_i)$ and $(\bang L_n)$ resp. In the case $A = p \bs B$, if $A$ is principal, then, by the requirement \ref{def_basic_bs_L} of Definition \ref{def_basic}, the last rule application in the basic derivation must be of the form
		$$
		\infer[(\bs L)]
		{
			\Gamma^\prime, p, p \bs B, \Psi \vdash a_L \cdot \okay
		}
		{
			\Gamma^\prime, B, \Psi \vdash a_L \cdot \okay
			&
			p \vdash p
		}
		$$
		
		This proves statement \ref{item_bottom-top_bs}. Now, let us prove statement \ref{item_bottom-top_nabla}. We clam that, if $\Delta_1, \nabla B, \Delta_2 \vdash a_L \cdot \okay$ is derivable where $\Delta_1, \Delta_2 = \Gamma,\Psi$, then, in any basic derivation of this sequent, $\nabla B$ is principal in the last rule application. The only difference in the proof of this claim with the cases considered above is that now the leftmost formula in $\Psi$, which is $b \bs E$, can be preceded by a primitive formula $a$, which is the rightmost formula of $\Gamma$. However, since $b$ is not contained in $\Gamma$, $b \ne a$ and hence $b \bs E$ cannot be principal. The formulas in $\Psi$ cannot be principal due to the same reasons as in the previous cases and so cannot be $a_L \cdot \okay$. Hence $\nabla B$ is principal.
		
		In what follows, if we consider the basic derivation of $\Theta_1, \nabla B, \Theta_2 \vdash a_L \cdot \okay$ from bottom to top, then there are several applications of the rules $(\nabla P_1)$ and $(\nabla P_2)$ preceded by an application of $(\nabla L)$. The sequent above these rule applications is of the form $\Xi_1, B, \Xi_2 \vdash a_L \cdot \okay$ such that $\Xi_1, \Xi_2 = \Theta_1,\Theta_2$; therefore it is derivable, as desired.
		
		Statement \ref{item_bottom-top_bs-2} is proved by induction on $m$. We can show in the same way as in the previous case that only $A$ can be principal. Then, by the requirement \ref{def_basic_bs_L} of Definition \ref{def_basic}, the last rule application in the basic derivation must be of the form
		$$
		\infer[(\bs L)]
		{
			\Theta_1^{\prime\prime}, p_m, p_m \bs B^\prime, \Theta_2 \vdash a_L \cdot \okay
		}
		{
			\Theta_1^{\prime\prime}, B^\prime, \Theta_2 \vdash a_L \cdot \okay
			&
			p_m \vdash p_m
		}
		$$
		
		Here $\Theta_1^{\prime\prime}, p_m = \Theta_1$ and $B^\prime = p_{m-1} \bs \dotsc \bs p_1 \bs B$. Note that the induction hypothesis can be applied to $\Theta_1^{\prime\prime}, B^\prime, \Theta_2 \vdash a_L \cdot \okay$ because $\Theta_1^{\prime\prime},\Theta_2$ is of the form $\Gamma^{\prime},\Psi$ where $\Gamma^\prime$ is obtained from $\Gamma$ by removing one occurrence of $p_m$. The induction hypothesis implies that $\Theta_1^{\prime\prime} = \Theta_1^\prime, p_1,\dotsc,p_{m-1}$ and $\Theta_1^{\prime}, B, \Theta_2 \vdash a_L \cdot \okay$ is derivable, as desired.
	\end{proof}
	
	This gives us a handy criterion of derivability of a sequent of the form (\ref{eq_seq_bottom-top}) or (\ref{eq_seq_bottom-top_2}).
	
	\begin{example}\label{example_BTA}
		As an example of using Lemma \ref{lemma_bottom-top} let us prove that the sequent 
		\begin{equation}\label{eq_seq_example_BTA}
			a_L, a_1^{\inp}, a_\Pi, \energy, \Energy_{0}, \Energy_1, \Energy_2 \vdash a_L \cdot \okay
		\end{equation}
		is derivable if and only if $a_L, a_1^{\inp}, \Energy_{\Pi}, \Energy_1, \Energy_2 \vdash a_L \cdot \okay$ is derivable. Recall that $\Energy_0 = \OKAY \wedge \energy \bs \left((a_\Sigma \bs \Energy_{\Sigma}) \wedge (a_\Pi \bs \Energy_{\Pi})\right)$, $\Energy_1 = (\okay \bs \okay) \wedge \energy \bs (\energy \cdot \bang \Energy_0)$, $\Energy_2 = (\okay \bs \okay) \wedge \energy \bs (\energy \cdot \bang \Energy_1)$. 
		\begin{enumerate}
			\item Let us apply the statement \ref{item_bottom-top_bs} of Lemma \ref{lemma_bottom-top} to (\ref{eq_seq_example_BTA}) taking $\Gamma = a_L, a_1^{\inp}, a_\Pi, \energy$, $A = \Energy_0$ and $\Psi = \Energy_1, \Energy_2$ (note that $\Psi$ is locked). Consequently, (\ref{eq_seq_example_BTA}) is derivable if and only if either $a_L, a_1^{\inp}, a_\Pi, \energy, \okay \bs \okay, \Energy_1, \Energy_2 \vdash a_L \cdot \okay$ is derivable or so is $a_L, a_1^{\inp}, a_\Pi, \energy, \energy \bs \Energy, \Energy_1, \Energy_2 \vdash a_L \cdot \okay$. However, the former sequent is not derivable according to Lemma \ref{lemma_bottom-top} (\ref{item_bottom-top_bs}) because its derivability would imply that $\okay$ stands right to the left from $\okay \bs \okay$.
			\item According to Lemma \ref{lemma_bottom-top} (\ref{item_bottom-top_bs}), the sequent $a_L, a_1^{\inp}, a_\Pi, \energy, \energy \bs \Energy, \Energy_1, \Energy_2 \vdash a_L \cdot \okay$ is derivable if and only if so is $a_L, a_1^{\inp}, a_\Pi, \Energy, \Energy_1, \Energy_2 \vdash a_L \cdot \okay$.
			\item Repeating the same argument for $\Energy$ we obtain that $a_L, a_1^{\inp}, a_\Pi, \Energy, \Energy_1, \Energy_2 \vdash a_L \cdot \okay$ is equiderivable with $a_L, a_1^{\inp}, \Energy_{\Pi}, \Energy_1, \Energy_2 \vdash a_L \cdot \okay$, which concludes the proof.
		\end{enumerate}
		Instead of writing such an unnecessary long explanation, we would like to present these reasonings in the following concise form:
		
		\begin{center}
			\begin{tabular}{ll}
				&\\[-11pt]
				$a_L, a_1^{\inp}, a_\Pi, \energy, \principal{\OKAY \wedge (\energy \bs \Energy)},\Energy_1, \Energy_2 \vdash a_L \cdot \okay$
				&
				\ref{item_bottom-top_wedge} $(\wedge)$
				\\[2pt]
				&\\[-11pt]
				$a_L, a_1^{\inp}, a_\Pi, \energy, \principal{\energy \bs \Energy} ,\Energy_1, \Energy_2 \vdash a_L \cdot \okay$
				&
				\ref{item_bottom-top_bs} $(\bs)$
				\\[2pt]
				&\\[-11pt]
				$a_L, a_1^{\inp}, a_\Pi, \principal{(a_\Sigma \bs \Energy_{\Sigma}) \wedge (a_\Pi \bs \Energy_{\Pi})} ,\Energy_1, \Energy_2 \vdash a_L \cdot \okay$
				&
				\ref{item_bottom-top_wedge} $(\wedge)$
				\\[2pt]
				&\\[-11pt]
				$a_L, a_1^{\inp}, a_\Pi, \principal{a_\Pi \bs \Energy_{\Pi}} ,\Energy_1, \Energy_2 \vdash a_L \cdot \okay$
				&
				\ref{item_bottom-top_bs} $(\bs)$
				\\[2pt]
				&\\[-11pt]
				$a_L, a_1^{\inp}, \Energy_{\Pi} ,\Energy_1, \Energy_2 \vdash a_L \cdot \okay$
				&
				\\[2pt]
			\end{tabular}
		\end{center}
		
		The intended meaning of these lines is the following: ``for each $k$, the sequent in the line $k$ is derivable if and only if so is the sequent in the line $(k+1)$, which is proved by the corresponding statement of Lemma \ref{lemma_bottom-top}''. The frame around a formula indicates that this formula plays the role of $A$ in the application of Lemma \ref{lemma_bottom-top}.
	\end{example}
	
	Using the bottom-top analysis (i.e. Lemma \ref{lemma_bottom-top}) we prove a number of useful facts.
	
	\begin{corollary}\label{corollary_wedge}
		Let $\Gamma \in \PrFm^\ast$ be a sequence of primitive formulas and let $\Psi$ be a locked sequence of formulas. Then, the sequent $\Gamma, \bigwedge\limits_{i=1}^n A_i, \Psi \vdash a_L \cdot \okay$ is derivable if and only if $\Gamma,A_i, \Psi \vdash a_L \cdot \okay$ is derivable for some $i \in \{1,\dotsc,n\}$.
	\end{corollary}
	
	The proof is by straightforward induction on $n$. 
	
	The following lemma shows us how $\Rule_{SR}$ models an application of a rule of the string rewriting system $SR$ (recall Definition \ref{def_Rule} where $\Rule_{SR}$ is introduced).
	
	\begin{lemma}\label{lemma_Rule}
		Consider the following sequent:
		\begin{equation}\label{eq_seq_before_Rule}
			\mathcal{U}, \go, \Rule_{SR}^l,\Psi \vdash a_L \cdot \okay
		\end{equation}
		Here $SR$ is a string rewriting system over the alphabet $\mathcal{A}$ and $\mathcal{U} \in \mathcal{A}^\ast$ is a sequence of symbols considered as primitive formulas; besides, $\Psi$ is a locked sequence of formulas. Then (\ref{eq_seq_before_Rule}) is derivable if and only if a sequent of the form
		\begin{equation}\label{eq_seq_after_Rule}
			\mathcal{W}, \go, \Psi \vdash a_L \cdot \okay
		\end{equation}
		is derivable for some $\mathcal{W} \in \mathcal{A}^\ast$ such that $\mathcal{U} \Rightarrow_{SR}^l \mathcal{W}$.
	\end{lemma}
	
	\begin{proof}
		Induction on $l$. If $l=0$, then $\mathcal{U} = \mathcal{W}$, and the statement trivially holds. To prove the induction step assume that $l>0$. Let us apply the bottom-top analysis.
		
		\begin{center}
			\begin{tabular}{ll}
				&\\[-11pt]
				$\mathcal{U}, \go, \principal{\Rule_{SR}}, \Rule_{SR}^{l-1},\Psi \vdash a_L \cdot \okay$
				&
				\ref{item_bottom-top_bs} $(\bs)$
				\\[2pt]
				&\\[-11pt]
				$\mathcal{U}, \principal{\bigwedge\limits_{r \in SR} (\nabla \fm(r) \cdot \wait \bs \go)}, \Rule_{SR}^{l-1},\Psi \vdash a_L \cdot \okay
				$
				&
				Corollary \ref{corollary_wedge}
				\\[2pt]
				&\\[-11pt]
				$\mathcal{U}, \principal{\nabla \fm(r) \cdot \wait \bs \go}, \Rule_{SR}^{l-1},\Psi \vdash a_L \cdot \okay$
				&
				\ref{item_bottom-top_cdot} $(\cdot)$
				\\[2pt]
				&\\[-11pt]
				$\mathcal{U}, \principal{\nabla \fm(r)}, \wait \bs \go, \Rule_{SR}^{l-1},\Psi \vdash a_L \cdot \okay$
				&
				\ref{item_bottom-top_nabla} $(\nabla)$
				\\[2pt]
				&\\[-11pt]
				$\Theta_1,\fm(r), \Theta_2 \vdash a_L \cdot \okay$
				&
				\\[2pt]
			\end{tabular}
		\end{center}
		
		Here $r \in SR$ is some rule and $\Theta_1,\Theta_2$ are sequences of formulas such that 
		\begin{equation}\label{eq_Phi1_Phi2}
			\Theta_1,\Theta_2 = \mathcal{U}, \wait \bs \go, \Rule_{SR}^{l-1},\Psi.
		\end{equation}
		
		If $r = (c_1 \dotsc c_m \to b_1 \dotsc b_n)$, then $\fm(r) = c_m \bs \dotsc \bs c_1 (b_1 \cdot \dotsc \cdot b_n \cdot \nabla \wait)$ (recall Definition \ref{def_fm}). Note that $\wait$ is not contained in $\mathcal{U}$. In what follows, one can apply statement \ref{item_bottom-top_bs-2} of Lemma \ref{lemma_bottom-top} and infer that $\Theta_1, \fm(r), \Theta_2 \vdash a_L \cdot \okay$ is derivable if and only if $\Theta_1$ equals $\Theta_1^\prime,c_1,\dotsc,c_m$ and $\Theta_1^\prime, b_1 \cdot \dotsc \cdot b_n \cdot \nabla \wait,\Theta_2 \vdash a_L \cdot \okay$ is derivable. By Corollary \ref{corollary_reversible_rules}, the latter sequent is equiderivable with 
		\begin{equation}\label{eq_seq_nabla_wait}
			\Theta_1^\prime, b_1 , \dotsc, b_n, \nabla \wait, \Theta_2 \vdash a_L \cdot \okay.
		\end{equation}
		Since $\Theta_1 = \Theta_1^\prime,c_1,\dotsc,c_m$, it follows from (\ref{eq_Phi1_Phi2}) that $\Theta_1$ is a prefix of $\mathcal{U}$. Consequently, $\Theta_2 = \Theta_1^{\prime\prime},\wait \bs \go, \Rule_{SR}^{l-1},\Psi$ where $\mathcal{U} = \Theta_1^{\prime},c_1,\dotsc,c_m,\Theta_1^{\prime\prime}$. Now let us apply the statement \ref{item_bottom-top_nabla} of Lemma \ref{lemma_bottom-top} and conclude that (\ref{eq_seq_nabla_wait}) is derivable if and only if so is
		\begin{equation}\label{eq_seq_Xi1_Xi2}
			\Xi_1,\wait,\Xi_2 \vdash a_L \cdot \okay
		\end{equation}
		for some $\Xi_1$ and $\Xi_2$ such that 
		$$
		\Xi_1,\Xi_2 = 
		\Theta_1^\prime, b_1 , \dotsc, b_n, \Theta_2 = 
		\Theta_1^\prime, b_1 , \dotsc, b_n, \Theta_1^{\prime\prime},\wait \bs \go, \Rule_{SR}^{l-1},\Psi.
		$$
		
		Which formula can be principal in the last rule application of a basic derivation of (\ref{eq_seq_Xi1_Xi2})? Only $\wait \bs \go$ can be principal because any formula in $\Rule_{SR}^{l-1},\Psi$ is not of the form $\wait \bs C$ nor it contains this formula as a conjunct. Therefore, (\ref{eq_seq_Xi1_Xi2}) is the sequent $\Theta_1^\prime, b_1 , \dotsc, b_n, \Theta_1^{\prime\prime},\wait,\wait \bs \go, \Rule_{SR}^{l-1},\Psi \vdash a_L \cdot \okay$. Applying the bottom-top analysis we conclude that it is equiderivable with
		\begin{equation}\label{eq_seq_W}
			\Theta_1^\prime, b_1 , \dotsc, b_n, \Theta_1^{\prime\prime},\go, \Rule_{SR}^{l-1},\Psi \vdash a_L \cdot \okay
		\end{equation}
		One can apply the induction hypothesis to (\ref{eq_seq_W}) and conclude that it is derivable if and only if (\ref{eq_seq_after_Rule}) is derivable for some $\mathcal{W} \in \mathcal{A}^\ast$ such that $\Theta_1^\prime, b_1 , \dotsc, b_n, \Theta_1^{\prime\prime}
		\Rightarrow_{SR}^{l-1}
		\mathcal{W}$. 
		
		This concludes the proof. Indeed, the sequent (\ref{eq_seq_before_Rule}) is derivable if and only if, for some rule $(c_1\dotsc c_m \to b_1 \dotsc b_n) \in SR$, it holds that
		$
		\mathcal{U} = \Theta_1^\prime, c_1 , \dotsc, c_m, \Theta_1^{\prime\prime} \Rightarrow_{SR} \Theta_1^\prime, b_1 , \dotsc, b_n, \Theta_1^{\prime\prime}
		\Rightarrow_{SR}^{l-1}
		\mathcal{W}.
		$
	\end{proof}
	
	The formula $\Rule_{SR}$ simulates one rule application in $SR$. In what follows, the formula $\bang \Rule_{SR}$, which is a subformula of $\Technical(SR,f)$ (recall Definition \ref{def_Technical}), is able to simulate a derivation in $SR$ of any length. This is what the following lemma is about.
	
	\begin{lemma}\label{lemma_Technical}
		Consider the following sequent:
		\begin{equation}\label{eq_seq_before_Technical}
			\mathcal{U}, \go, \Technical(SR,f),\Psi \vdash a_L \cdot \okay
		\end{equation}
		Here $SR$ is a string rewriting system over an alphabet $\mathcal{A}$, $f:\{a_1,a_2\} \to \Fm$ is a function where $\{a_1,a_2\} \subseteq \mathcal{A}$ and $\mathcal{U} \in \mathcal{A}^\ast$; besides, $\Psi$ is a locked sequence of formulas. Then (\ref{eq_seq_before_Technical}) is derivable if and only if there exists $\mathcal{W} \in \mathcal{A}^\ast$ and $i \in \{1,2\}$ such that $\mathcal{U} a_R \Rightarrow_{SR}^\ast \mathcal{W} a_i \final$ and the sequent $\mathcal{W}, f(a_i), \Psi \vdash a_L \cdot \okay$ is derivable. 
	\end{lemma}
	
	\begin{proof}
		We apply the bottom-top analysis:
		
		\begin{center}
			\begin{tabular}{ll}
				&\\[-11pt]
				$\mathcal{U}, \go, \principal{\Technical(SR,f)},\Psi \vdash a_L \cdot \okay$
				&
				\ref{item_bottom-top_bs} $(\bs)$
				\\[2pt]
				&\\[-11pt]
				$\mathcal{U}, \principal{a_R \cdot \go \cdot 
				\bang\Rule_{SR} \cdot \go \bs \final \bs \left( a_1 \bs f(a_1) \wedge a_2 \bs f(a_2) \right)},\Psi \vdash a_L \cdot \okay$
				&
				\ref{item_bottom-top_cdot} $(\cdot)$
				\\[2pt]
				&\\[-11pt]
				$\mathcal{U}, a_R, \go, 
				\principal{\bang\Rule_{SR}}, \go \bs \final \bs \left( a_1 \bs f(a_1) \wedge a_2 \bs f(a_2) \right),\Psi \vdash a_L \cdot \okay$
				&
				\ref{item_bottom-top_bang} $(\bang)$
				\\[2pt]
				&\\[-11pt]
				$\mathcal{U}, a_R, \go, 
				\Rule_{SR}^l, \go \bs \final \bs \left( a_1 \bs f(a_1) \wedge a_2 \bs f(a_2) \right),\Psi \vdash a_L \cdot \okay$
				&
				\\[2pt]
			\end{tabular}
		\end{center}
		
		Summing up, (\ref{eq_seq_before_Technical}) is derivable if and only if the last sequent is derivable for some $l$. Lemma \ref{lemma_Rule} implies that the latter sequent is derivable if and only if there is $\mathcal{W}^{\prime\prime} \in \mathcal{A}^\ast$ such that $\mathcal{U} a_R \Rightarrow_{SR}^l \mathcal{W}^{\prime\prime}$ and such that the sequent $\mathcal{W}^{\prime\prime}, \go, \go \bs \final \bs \left( a_1 \bs f(a_1) \wedge a_2 \bs f(a_2) \right),\Psi \vdash a_L \cdot \okay$ is derivable. 
		
		Let us continue the bottom-top analysis. Statement \ref{item_bottom-top_bs} of Lemma \ref{lemma_bottom-top} implies that the latter sequent is derivable if and only if $\mathcal{W}^{\prime\prime} = \mathcal{W}^\prime \final$ and the sequent
		$
		\mathcal{W}^\prime, a_1 \bs f(a_1) \wedge a_2 \bs f(a_2),\Psi \vdash a_L \cdot \okay
		$
		is derivable. By Lemma \ref{lemma_bottom-top} (\ref{item_bottom-top_wedge}), derivability of this sequent is equivalent to derivability of the sequent
		$
		\mathcal{W}^\prime, a_i \bs f(a_i),\Psi \vdash a_L \cdot \okay
		$
		for some $i \in \{1,2\}$. Again, by Lemma \ref{lemma_bottom-top} (\ref{item_bottom-top_bs}), this is the case if and only if $\mathcal{W}^\prime = \mathcal{W} a_i$ and the sequent
		$
		\mathcal{W}, f(a_i),\Psi \vdash a_L \cdot \okay
		$
		is derivable. It remains to note that $\mathcal{W}^{\prime\prime} = \mathcal{W} a_i \final$.
	\end{proof}
	
	The following lemma explains the idea behind the formula $\OKAY = \okay \bs \okay$.
	
	\begin{lemma}\label{lemma_ok}
		The sequent $a_L, \okay, \Psi \vdash a_L \cdot \okay$ where $\Psi$ consists of formulas of the form $\OKAY \wedge (p \bs B)$ is derivable.
	\end{lemma}
	
	It is proved by straightforward induction. Informally, this lemma holds because the sequent of interest is obtained from the sequent $a_L,\okay,(\okay \bs \okay)^c \vdash a_L \cdot \okay$ (where $c = \vert \Psi \vert$) by applying $(\wedge L_2)$ several times; the latter sequent is obviously derivable.
	
	The final auxiliary lemma is concerned with the formula $[\go]^\rightarrow p = \go \bs (p \cdot \go)$.
	
	\begin{lemma}\label{lemma_rightarrow}
		Let $\Gamma \in \PrFm^\ast$ be a sequence of primitive formulas, let $p \in \PrFm$ be a primitive formula, let $c \in \Nat$, and let $\Psi$ be locked. Then the sequent $\Gamma,\go,([ \go ]^\rightarrow p)^c, \Psi \vdash a_L \cdot \okay$ is derivable if and only if so is $\Gamma,p^c,\go, \Psi \vdash a_L \cdot \okay$.
	\end{lemma}
	
	\begin{proof}
		Induction on $c$. The case $c = 0$ is trivial. To prove the induction step assume that $c>0$. Let us apply the bottom-top analysis.
		
		\begin{center}
			\begin{tabular}{ll}
				&\\[-11pt]
				$\Gamma,\go,\principal{[ \go ]^\rightarrow p},([ \go ]^\rightarrow p)^{c-1}, \Psi \vdash a_L \cdot \okay$
				&
				\ref{item_bottom-top_bs} $(\bs)$
				\\[2pt]
				&\\[-11pt]
				$\Gamma,\principal{p \cdot \go},([ \go ]^\rightarrow p)^{c-1}, \Psi \vdash a_L \cdot \okay$
				&
				\ref{item_bottom-top_cdot} $(\cdot)$
				\\[2pt]
				&\\[-11pt]
				$\Gamma,p, \go,([ \go ]^\rightarrow p)^{c-1}, \Psi \vdash a_L \cdot \okay$
				&\\[2pt]
			\end{tabular}
		\end{center}
		The induction hypothesis completes the proof.
	\end{proof}

	\subsection{Proof of Lemma \ref{lemma_main}}\label{ssec_proof}
	
	We are ready to prove Lemma \ref{lemma_main}. Recall that we study derivability of the sequent (\ref{eq_seq_main}) of the form $a_L, a_1^{\inp}, a_X, \energy, \Energy_{h_M}, \dotsc, \Energy_{h_1} \vdash a_L \cdot \okay$ for $\inp = \langle \langle X,\pi(\alpha),i,e \rangle,\langle n_1,\dotsc,n_i \rangle\rangle$. We aim to prove that 1.~ its derivability implies that $\inp \in \SATOO$ and that 2.~ if $\inp \in \SATOO$ and $\omega^{h_1}+\dotsc+\omega^{h_M} > \alpha$, then (\ref{eq_seq_main}) is derivable.
	
	\begin{proof}
		We prove both statements together by induction on $\beta = \omega^{h_1}+\dotsc+\omega^{h_M}$. The base case is where $\beta = 0$, i.e. $M = 0$. In this case, obviously, (\ref{eq_seq_main}) is not derivable. Since $0 \not> \alpha$, both statements of Lemma \ref{lemma_main} trivially hold. Now, assume that $\beta > 0$.
				
		\textbf{Case 1.} $h_M > 0$ (i.e. $\beta$ is limit). Then $\Energy_{h_M} = \OKAY \wedge \energy \bs \left(\energy \cdot \bang \Energy_{h_M-1}\right)$. Let us apply the bottom-top analysis.
		
		\begin{center}
			\begin{tabular}{ll}
				&\\[-11pt]
				$a_L, a_1^{\inp}, a_X, \energy, \principal{\Energy_{h_M}}, \dotsc, \Energy_{h_1} \vdash a_L \cdot \okay$
				&
				\ref{item_bottom-top_wedge} $(\wedge)$
				\\[2pt]
				&\\[-11pt]
				$a_L, a_1^{\inp}, a_X, \energy, \principal{\energy \bs \left(\energy \cdot \bang \Energy_{h_M-1}\right)}, \Energy_{h_{M-1}}, \dotsc, \Energy_{h_1} \vdash a_L \cdot \okay$
				&
				\ref{item_bottom-top_bs} $(\bs)$
				\\[2pt]
				&\\[-11pt]
				$a_L, a_1^{\inp}, a_X, \principal{\energy \cdot \bang \Energy_{h_M-1}}, \Energy_{h_{M-1}}, \dotsc, \Energy_{h_1} \vdash a_L \cdot \okay$
				&
				\ref{item_bottom-top_cdot} $(\cdot)$
				\\[2pt]
				&\\[-11pt]
				$a_L, a_1^{\inp}, a_X, \energy, \principal{\bang \Energy_{h_M-1}}, \Energy_{h_{M-1}}, \dotsc, \Energy_{h_1} \vdash a_L \cdot \okay$
				&
				\ref{item_bottom-top_bang} $(\bang)$
				\\[2pt]	
			\end{tabular}
			\begin{equation}\label{eq_seq_BTA1}
				\mbox{for some $l$,} \quad a_L, a_1^{\inp}, a_X, \energy, \Energy_{h_M-1}^{l}, \Energy_{h_{M-1}}, \dotsc, \Energy_{h_1} \vdash a_L \cdot \okay
			\end{equation}
		\end{center}
		
		We can apply the induction hypothesis to (\ref{eq_seq_BTA1}). Let $\beta^\prime_l \eqdef \omega^{h_1}+\dotsc+\omega^{h_{M-1}}+\omega^{h_M-1}\cdot l$. By the induction hypothesis:
		\begin{enumerate}
			\item if (\ref{eq_seq_BTA1}) is derivable, then $\inp \in \SATOO$;
			\item if $\inp \in \SATOO$ and $\beta^\prime_l > \alpha$, then (\ref{eq_seq_BTA1}) is derivable.
		\end{enumerate}
		Therefore, if (\ref{eq_seq_main}) is derivable, then so is (\ref{eq_seq_BTA1}) (for some $l$) and thus $\inp \in \SATOO$, as desired. Conversely, let $\inp \in \SATOO$ and let $\beta > \alpha$. Since $\beta = \sup_{l \in \Nat} \beta^\prime_l$, it holds that $\beta^\prime_l > \alpha$ for some $l$. In what follows, (\ref{eq_seq_BTA1}) is derivable for this $l$ and hence so is (\ref{eq_seq_main}).
		
		\textbf{Case 2.} $h_M = 0$. Then $\Energy_{h_M} = \OKAY \wedge \energy \bs ((a_\Sigma \bs \Energy_{\Sigma}) \wedge (a_\Pi \bs \Energy_{\Pi}))$. 
		
		\textbf{Case 2a.} $X = \Sigma$. Again, we use the bottom-top analysis (cf. Example \ref{example_BTA}):
		
		\begin{center}
			\begin{tabular}{ll}
				&\\[-11pt]
				$a_L, a_1^{\inp}, a_\Sigma, \energy, \principal{\Energy_{h_M}},\Energy_{h_{M-1}}, \dotsc, \Energy_{h_1} \vdash a_L \cdot \okay$
				&
				\ref{item_bottom-top_wedge} $(\wedge)$
				\\[2pt]
				&\\[-11pt]
				$a_L, a_1^{\inp}, a_\Sigma, \energy, \principal{\energy \bs \Energy} ,\Energy_{h_{M-1}}, \dotsc, \Energy_{h_1} \vdash a_L \cdot \okay$
				&
				\ref{item_bottom-top_bs} $(\bs)$
				\\[2pt]
				&\\[-11pt]
				$a_L, a_1^{\inp}, a_\Sigma, \principal{(a_\Sigma \bs \Energy_{\Sigma}) \wedge (a_\Pi \bs \Energy_{\Pi})} ,\Energy_{h_{M-1}}, \dotsc, \Energy_{h_1} \vdash a_L \cdot \okay$
				&
				\ref{item_bottom-top_wedge} $(\wedge)$
				\\[2pt]
				&\\[-11pt]
				$a_L, a_1^{\inp}, a_\Sigma, \principal{a_\Sigma \bs \Energy_{\Sigma}} ,\Energy_{h_{M-1}}, \dotsc, \Energy_{h_1} \vdash a_L \cdot \okay$
				&
				\ref{item_bottom-top_bs} $(\bs)$
				\\[2pt]
				&\\[-11pt]
				$a_L, a_1^{\inp}, \principal{\Energy_{\Sigma}} ,\Energy_{h_{M-1}}, \dotsc, \Energy_{h_1} \vdash a_L \cdot \okay
				$
				&
				\ref{item_bottom-top_cdot} $(\cdot)$
				\\[2pt]
			\end{tabular}
			\begin{equation}\label{eq_seq_BTA2a}
				a_L, a_1^{\inp}, \go, \Technical\left(\sr_0,f_0\right), \mathcal{B} \vdash a_L \cdot \okay
			\end{equation}
		\end{center}
		Here $\mathcal{B}= \OKAY \wedge [ \go ]^?\bang [ \go ]^\rightarrow a_2,
		\OKAY \wedge \Technical\left(\sr_1,f_\Sigma\right), \Energy_{h_{M-1}}, \dotsc, \Energy_{h_1}$ (recall that $\Energy_\Sigma = \go \cdot \Technical\left(\sr_0,f_0\right) \cdot 
		\left(\OKAY \wedge [ \go ]^?\bang [ \go ]^\rightarrow a_2 \right) \cdot
		\left(\OKAY \wedge \Technical\left(\sr_1,f_\Sigma\right) \right)$). It follows from Lemma \ref{lemma_Technical} applied to (\ref{eq_seq_BTA2a}) that it is derivable if and only if there exist $\mathcal{W}$ and $j \in \{1,2\}$ such that 
		\begin{equation}\label{eq_SR0}
			a_L a_1^{\inp} a_R \Rightarrow_{\sr_0}^\ast \mathcal{W} a_j \final
		\end{equation}
		and the sequent $\mathcal{W}, f_0(a_j), \mathcal{B} \vdash a_L \cdot \okay$ is derivable. The fact (\ref{eq_SR0}) is equivalent to 
		$$
		\mathcal{W} a_j = a_L \mathcal{F}_0\left(a_1^{\inp}\right)
		$$
		since $\sr_0$ $(a_L,a_R,\final)$-implements $\mathcal{F}_0$ (according to Corollary \ref{corollary_Church-Turing}, $a_L u a_R \Rightarrow_{\sr_0}^\ast v \final$ if and only if $v = a_L \mathcal{F}_0(u)$). Note that $\mathcal{F}_0(w)$ ends by either $a_1$ or $a_2$. Let us consider each of the two subcases.
		
		\textbf{Subcase i.} $\mathcal{F}_0\left(a_1^{\inp}\right) = a_1$. This happens, according to the definition of $\mathcal{F}_0$ (see Definition \ref{def_Energy}) iff $\alpha = 0$ and $\inp \in \SATOO$. The sequent $\mathcal{W}, f_0(a_1), \mathcal{B} \vdash a_L \cdot \okay$ equals $a_L, \okay, \mathcal{B} \vdash a_L \cdot \okay$ (since $f_0(a_1) = \okay$). Lemma \ref{lemma_ok} implies that it is derivable.
		
		\textbf{Subcase ii.} $\mathcal{F}_0\left(a_1^{\inp}\right) = a_1^{\inp} a_2$. This happens iff $\alpha > 0$. The sequent $\mathcal{W}, f_0(a_2), \mathcal{B} \vdash a_L \cdot \okay$ equals $a_L, a_1^{\inp}, \go, \mathcal{B} \vdash a_L \cdot \okay$.
		
		Concluding the above reasonings, if $\alpha = 0$, then (\ref{eq_seq_main}) is derivable if and only if so is (\ref{eq_seq_BTA2a}), which holds if and only if $\inp \in \SATOO$. For $\alpha > 0$, we have proved that (\ref{eq_seq_main}) is equiderivable with the sequent $a_L, a_1^{\inp}, \go, \mathcal{B} \vdash a_L \cdot \okay$. We proceed with the bottom-top analysis of the latter sequent. Recall that 
		$\mathcal{B} = \OKAY \wedge [ \go ]^?\bang [ \go ]^\rightarrow a_2, \mathcal{B}_2$ where $\mathcal{B}_2 = \OKAY \wedge \Technical\left(\sr_1,f_\Sigma\right), \Energy_{h_{M-1}}, \dotsc, \Energy_{h_1}$.
		
		\begin{center}
			\begin{tabular}{ll}
				&\\[-11pt]
				$a_L, a_1^{\inp}, \go, \principal{\OKAY \wedge [ \go ]^?\bang [ \go ]^\rightarrow a_2} , \mathcal{B}_2 \vdash a_L \cdot \okay$
				&
				\ref{item_bottom-top_wedge} $(\wedge)$
				\\[2pt]
				&\\[-11pt]
				$a_L, a_1^{\inp}, \go, \principal{[ \go ]^?\bang [ \go ]^\rightarrow a_2} , \mathcal{B}_2 \vdash a_L \cdot \okay$
				&
				\ref{item_bottom-top_bs} $(\bs)$
				\\[2pt]
				&\\[-11pt]
				$a_L, a_1^{\inp}, \principal{\go \cdot \bang [ \go ]^\rightarrow a_2} , \mathcal{B}_2 \vdash a_L \cdot \okay$
				&
				\ref{item_bottom-top_cdot} $(\cdot)$
				\\[2pt]
				&\\[-11pt]
				$a_L, a_1^{\inp}, \go, \principal{\bang [ \go ]^\rightarrow a_2} , \mathcal{B}_2 \vdash a_L \cdot \okay$
				&
				\ref{item_bottom-top_bang} $(\bang)$
				\\[2pt]
				&\\[-11pt]
				$a_L, a_1^{\inp}, \go, \left([ \go ]^\rightarrow a_2\right)^c , \mathcal{B}_2 \vdash a_L \cdot \okay$
				&
				Lemma \ref{lemma_rightarrow}
				\\[2pt]
				&\\[-11pt]
				$a_L, a_1^{\inp}, a_2^{c}, \go, \principal{\OKAY \wedge \Technical\left(\sr_1,f_\Sigma\right)}, \Energy_{h_{M-1}}, \dotsc, \Energy_{h_1} \vdash a_L \cdot \okay$
				&
				\ref{item_bottom-top_wedge} $(\wedge)$
				\\[2pt]
			\end{tabular}
			\begin{equation}\label{eq_seq_Technical_Sigma}
				a_L, a_1^{\inp}, a_2^c, \go, \Technical\left(\sr_1,f_\Sigma\right) , \Energy_{h_{M-1}}, \dotsc, \Energy_{h_1} \vdash a_L \cdot \okay
			\end{equation}
		\end{center}
		
		Lemma \ref{lemma_Technical} applied to (\ref{eq_seq_Technical_Sigma}) implies that this sequent is derivable if and only if there exist $\mathcal{V}$ and $b \in \{a_1,a_2\}$ such that 
		\begin{equation}\label{eq_SR1}
			a_L a_1^{\inp} a_2^c a_R \Rightarrow_{\sr_1}^\ast \mathcal{V} b \final
		\end{equation}
		and the sequent $\mathcal{V}, f_\Sigma(b), \Energy_{h_{M-1}}, \dotsc, \Energy_{h_1} \vdash a_L \cdot \okay$ is derivable. The fact (\ref{eq_SR1}) is equivalent to the fact that $\mathcal{V} b = a_L  \mathcal{F}_1\left(a_1^{\inp} a_2^c\right)$
		since $\sr_1$ $(a_L,a_R,\final)$-implements $\mathcal{F}_1$.
		
		\textbf{Subcase A.} $\mathcal{F}_1\left(a_1^{\inp} a_2^c\right) = a_1^{\langle \langle \Pi, \pi(\alpha^\prime), i+j, e^\prime \rangle, \langle n_1,\dotsc,n_{i+j} \rangle \rangle} a_2$. This happens if and only if $c = \langle c^\prime, y, \langle n_{i+1}, \dotsc,n_{i+j}\rangle \rangle$ where $c^\prime = \langle \pi(\alpha^\prime),j,e^\prime \rangle$ for $\alpha^\prime < \alpha$ such that $\Halt(c^\prime,e,y)$ (Definition \ref{def_Energy}). In this case, since $f_\Sigma(a_2) = a_\Pi \cdot \energy$, the sequent (\ref{eq_seq_Technical_Sigma}) is equiderivable with 
		\begin{equation}\label{eq_seq_BTA2aA}
			a_L, a_1^{\langle \langle \Pi, \pi(\alpha^\prime), i+j, e^\prime \rangle, \langle n_1,\dotsc,n_{i+j} \rangle \rangle}, a_\Pi , \energy, \Energy_{h_{M-1}}, \dotsc, \Energy_{h_1} \vdash a_L \cdot \okay.
		\end{equation}

		\textbf{Subcase B.} Otherwise, $\mathcal{F}_1\left(a_1^{\inp} a_2^c\right) = a_1$. In this case, (\ref{eq_seq_Technical_Sigma}) is derivable if and only if so is $a_L, \fail, \Energy_{h_{M-1}}, \dotsc, \Energy_{h_1} \vdash a_L \cdot \okay$. However, the latter sequent is not derivable.
		
		Now, we can apply the induction hypothesis to (\ref{eq_seq_BTA2aA}) because $\beta^\prime = \omega^{h_1}+\dotsc+\omega^{h_{M-1}} < \beta$. We obtain the following.
		\begin{enumerate}
			\item If (\ref{eq_seq_main}) is derivable, then the sequent of the form (\ref{eq_seq_BTA2aA}) is derivable for some $\inp^\prime = \langle \langle \Pi, \pi(\alpha^\prime), i+j, e^\prime \rangle, \langle n_1,\dotsc,n_{i+j} \rangle \rangle$. By the induction hypothesis, this means that  $\inp^\prime \in \SATOO$. The latter implies that $\Nat \vDash \varphi^{\Sigma_{\alpha}}_{e,i}(n_1,\dotsc,n_i)$ because this formula is a disjunction of formulas $\exists x_{i+1} \dotsc \exists x_{i+j} \; \varphi^{\Pi_{\alpha^\prime}}_{e^\prime,i+j}(n_1,\dotsc,n_i,x_{i+1}, \dotsc, x_{i+j})$. Therefore, $\inp \in \SATOO$.
			\item If $\langle \langle \Sigma,\pi(\alpha),i,e \rangle,\langle n_1,\dotsc,n_i\rangle\rangle \in \SATOO$, then there exists a triple of the form $c^\prime = \langle \pi(\alpha^\prime),j,e^\prime\rangle$ such that $\alpha^\prime < \alpha$, $c^\prime \in W_e$ (i.e. $\Halt(c^\prime,e,y)$ holds for some $y$) and $\inp^\prime = \langle \langle \Pi, \pi(\alpha^\prime), i+j, e^\prime \rangle, \langle n_1,\dotsc,n_{i+j} \rangle \rangle \in \SATOO$. The inequalities $\alpha^\prime < \alpha$ and $\alpha < \beta^\prime + 1$ imply that $\alpha^\prime < \beta^\prime$. Thus one can apply the induction hypothesis and conclude that (\ref{eq_seq_BTA2aA}) is derivable for such $\inp^\prime$, which implies derivability of (\ref{eq_seq_main}).
		\end{enumerate}
		
		\textbf{Case 2b.} $X = \Pi$. The proof is of the same structure as that for Case 2a so we shall only outline the main stages of the argument without repeating the same details.
		
		\begin{enumerate}
			\item The sequent $a_L, a_1^{\inp}, a_\Pi, \energy, \Energy_{h_M}, \dotsc, \Energy_{h_1} \vdash a_L \cdot \okay$ is equiderivable with the sequent 
			\begin{equation}\label{eq_seq_2b_before_Technical_0}
				a_L, a_1^{\inp}, \go, \Technical\left(\sr_0,f_0\right), \mathcal{B}^\prime \vdash a_L \cdot \okay
			\end{equation}
			where
			$
			\mathcal{B}^\prime = \OKAY \wedge [ \go ]^\rightarrow a_2^\ast, 
			\OKAY \wedge \Technical(\sr_1,f_\Pi) , \Energy_{h_{M-1}}, \dotsc, \Energy_{h_1}
			$
			(according to the definition of $\Energy_{\Pi}$). This is proved using the bottom-up analysis.
			\item If $\alpha = 0$, then the proof is the same as for the case $X=\Sigma$. \item If $\alpha > 0$, then (\ref{eq_seq_2b_before_Technical_0}) is equiderivable with the following sequent:
			\begin{equation}\label{eq_seq_2b_after_Technical_0}
				a_L, a_1^{\inp}, \go, \mathcal{B}^\prime \vdash a_L \cdot \okay
			\end{equation}
			We proceed with the bottom-top analysis of (\ref{eq_seq_2b_after_Technical_0}). Let us denote by $\mathcal{B}^\prime_2$ the sequence $\OKAY \wedge \Technical(\sr_1,f_\Pi) , \Energy_{h_{M-1}}, \dotsc, \Energy_{h_1}$.	
			
			\begin{center}
				\begin{tabular}{ll}
					&\\[-11pt]
					$a_L, a_1^{\inp}, \go, \principal{\OKAY \wedge [ \go ]^\rightarrow a_2^\ast} , \mathcal{B}^\prime_2 \vdash a_L \cdot \okay$
					&
					\ref{item_bottom-top_wedge} $(\wedge)$
					\\[2pt]
					&\\[-11pt]
					$a_L, a_1^{\inp}, \go, \principal{[ \go ]^\rightarrow a_2^\ast} , \mathcal{B}^\prime_2 \vdash a_L \cdot \okay$
					&
					\ref{item_bottom-top_bs} $(\bs)$
					\\[2pt]
					&\\[-11pt]
					$a_L, a_1^{\inp}, \principal{a_2^\ast \cdot \go} , \mathcal{B}^\prime_2 \vdash a_L \cdot \okay$
					&
					\ref{item_bottom-top_cdot} $(\cdot)$
					\\[2pt]
					&\\[-11pt]
					$a_L, a_1^{\inp}, \principal{a_2^\ast}, \go, \mathcal{B}^\prime_2 \vdash a_L \cdot \okay$
					&
					Corollary \ref{corollary_reversible_rules}
					\\[2pt]
					&\\[-11pt]
					$\forall c \in \Nat \qquad a_L, a_1^{\inp}, a_2^c, \go, \mathcal{B}^\prime_2 \vdash a_L \cdot \okay$
					&
					\\[2pt]
				\end{tabular}
			\end{center}
			
			Consequently, Lemma \ref{lemma_bottom-top} (\ref{item_bottom-top_wedge}) implies that (\ref{eq_seq_2b_after_Technical_0}) is derivable if and only if, for each $c \in \Nat$, the following sequent is derivable:
			\begin{equation}\label{eq_seq_Technical_Pi}
				a_L, a_1^{\inp}, a_2^c, \go, \Technical(\sr_1,f_\Pi) , \Energy_{h_{M-1}}, \dotsc, \Energy_{h_1} \vdash a_L \cdot \okay.
			\end{equation}
			\item Derivability of (\ref{eq_seq_Technical_Pi}) is analysed in the same way as that of (\ref{eq_seq_Technical_Sigma}). Namely, the following two subcases arise.
			
			\begin{enumerate}
				\item If $c = \langle c^\prime, y, \langle n_{i+1}, \dotsc,n_{i+j}\rangle \rangle$ where $c^\prime = \langle \pi(\alpha^\prime),j,e^\prime \rangle$ for $\alpha^\prime < \alpha$ such that $\Halt(c^\prime,e,y)$ holds, then (\ref{eq_seq_Technical_Pi}) is equiderivable with the sequent 
				\begin{equation}\label{eq_seq_BTA2b}
					a_L, a_1^{\langle \langle \Sigma, \pi(\alpha^\prime), i+j, e^\prime \rangle, \langle n_1,\dotsc,n_{i+j} \rangle \rangle}, a_\Sigma , \energy, \Energy_{h_{M-1}}, \dotsc, \Energy_{h_1} \vdash a_L \cdot \okay.
				\end{equation}
				\item If $c$ is not of the form described above, then (\ref{eq_seq_Technical_Pi}) is equiderivable with the sequent 
				$
				a_L, \okay, \Energy_{h_{M-1}}, \dotsc, \Energy_{h_1} \vdash a_L \cdot \okay.
				$
				By Lemma \ref{lemma_ok}, the latter is derivable.
			\end{enumerate}
		\end{enumerate}
		Summarizing these reasonings, for $X = \Pi$ and $\alpha > 0$, the sequent (\ref{eq_seq_main}) is derivable if and only if, for all triples $c^\prime = \langle \pi(\alpha^\prime),j,e^\prime \rangle$ such that $\alpha^\prime < \alpha$, $c^\prime \in W_e$ and for all $n_{i+1},\dotsc,n_{i+j} \in \Nat$, the sequent (\ref{eq_seq_BTA2b}) is derivable. 
		\begin{enumerate}
			\item By the induction hypothesis, if (\ref{eq_seq_BTA2b}) is derivable, then $\varphi^{\Sigma_{\alpha^\prime}}_{e^\prime,i+j}(n_1,\dotsc,n_{i+j})$ is true. The fact that this holds for all the triples of the form $c^\prime = \langle \pi(\alpha^\prime),j,e^\prime \rangle$ such that $\alpha^\prime < \alpha$, $c^\prime \in W_e$ and for all the tuples $\langle n_{i+1},\dotsc,n_{i+j} \rangle \in \Nat$ is equivalent to the fact that $\Nat \vDash \varphi^{\Pi_{\alpha}}_{e,i}(n_1,\dotsc,n_i)$ because $\varphi^{\Pi_{\alpha}}_{e,i}(n_1,\dotsc,n_i)$ is a conjunction of formulas of the form $\forall x_{i+1} \dotsc \forall x_{i+j} \; \varphi^{\Sigma_{\alpha^\prime}}_{e^\prime,i+j}(n_1,\dotsc,n_i,x_{i+1}, \dotsc, x_{i+j})$. 
			\item Conversely, if $\Nat \vDash \varphi^{\Pi_{\alpha}}_{e,i}(n_1,\dotsc,n_i)$, then $\varphi^{\Sigma_{\alpha^\prime}}_{e^\prime,i+j}(n_1,\dotsc,n_i,n_{i+1},\dotsc,n_{i+j})$ is true for all triples $c^\prime = \langle \pi(\alpha^\prime),j,e^\prime \rangle$ such that $\alpha^\prime < \alpha$, $c^\prime \in W_e$ and for all $n_{i+1},\dotsc,n_{i+j}$. Besides, for each such triple, $\beta^\prime > \alpha^\prime$ so one can apply the induction hypothesis and conclude that all the sequents of the form (\ref{eq_seq_BTA2b}) are derivable, hence so is (\ref{eq_seq_main}).
		\end{enumerate} 
	\end{proof}
	
	Therefore, we have proved correctness of the construction presented in Lemma \ref{lemma_main}.
	
	\begin{corollary}\label{corollary_SATOO_to_ACTMult}
		$\SATOO$ is one-one reducible to derivability in $\ACTMult$.
	\end{corollary}
	\begin{proof}
		The reduction is as follows: 
		$$
		\inp \mapsto \mathrm{seq}(\inp) \eqdef a_L, a_1^{\inp}, a_X, \energy, \Energy_{h_M}, \dotsc, \Energy_{h_1} \vdash a_L \cdot \okay.
		$$
		Here $\omega^{h_1} + \dotsc + \omega^{h_M} = \alpha + 1$ if $\inp = \langle \langle X,\pi(\alpha),i,e \rangle,\langle n_1,\dotsc,n_i \rangle\rangle$ (in order to make the reduction total assume that otherwise $M = 0$). Then, according to Lemma \ref{lemma_main}, $\inp \in \SATOO$ if and only if $\mathrm{seq}(\inp)$ is derivable. Note that $\mathrm{seq}$ is injective.
	\end{proof}
	
	\subsection{Reduction of Derivability in $\ACTMult$ to $\SATOO$}\label{ssec_ACTMult_to_SATOO}
	
	What remains is to show that the derivability problem for $\ACTMult$ is one-one reducible to $\SATOO$. We shall construct a computable function $F :\PN \to \Nat$ such that $F(p)$ is an index of a formula $\Der_p(x_1)$ with the following property. If $\rho(\Gamma \vdash C) \le \pi^{-1}(p)$, then $\Der_p(\ulcorner \Gamma \vdash C \urcorner)$ is true if and only if $\Gamma \vdash C$ is derivable in $\ACTMult$. Recall that $\ulcorner \Gamma \vdash C \urcorner$ is a G\"odel number of $\Gamma \vdash C$; we assume that $0$ is not a G\"odel number of any sequent. Hereinafter, we consider the logic $\ACTMult$ including constants, $/$, and $\vee$.
	
	Assume that we know that the result of some rule application is $A_1, \dotsc,A_M \vdash A_0$ and we would like to know how this rule application could look like. It can be described by a number $m \in \{0,\dotsc,M\}$ which says that the formula $A_m$ is principal (this also gives one information about which rule is applied) and by an additional parameter $l$, whose meaning depends on the rule applied. E.g. if the rule is $(\bs L)$ (as defined in Section \ref{ssec_ACTMult}), then $l \eqdef \vert \Pi \vert$; if the rule is $(\cdot R)$, then $l \eqdef \vert \Gamma \vert$; if the rule is $(\bang L_n)$, then $l \eqdef n$; etc. Thus the single number $t = \langle m,l \rangle$ defines a rule application.
	\begin{example}
		Assume that the sequent $p, p \bs q, q \bs r, r \bs s \vdash s$ is known to be the result of some rule application. If $t = \langle 3, 2 \rangle = 2^3 \cdot 3^2$, then this rule application looks as follows:
		$$
		\infer[(\bs L)]
		{
			p, p \bs q, q \bs r, r \bs s \vdash s
		}
		{
			r, r \bs s \vdash s
			&
			p, p \bs q \vdash q
		}
		$$
		
	\end{example}
	To handle permutation rules $(\nabla P_i)$ one should consider generalized rules; then, the parameter $l$ must also contain information about whether a formula of the form $\nabla B$ moves to the left or to the right and how many formulas it jumps over.
	
	Finally, if $t \in \Nat$ defines a rule application with the conclusion $A_1, \dotsc,A_M \vdash A_0$, then let $\Premise(\ulcorner A_1, \dotsc,A_M \vdash A_0 \urcorner,t,k)$ be the code for the $(k+1)$-th premise of this rule application (for $k \in \Nat$). If $(k+1)$ is greater that the total number of premises, then let $\Premise(\ulcorner A_1, \dotsc,A_M \vdash A_0 \urcorner,t,k) \eqdef \ulcorner \mathbf{1} \vdash \mathbf{1} \urcorner$. If $c$ is not a G\"odel number of a sequent or $t$ does not define a rule application, then let $\Premise(c,t,k) \eqdef 0$. Clearly, the function $\Premise$ is computable.
	Finally, let us recursively define $\Der_{p}(x_1)$ as follows:
	$$
	\Der_{p}(x_1) \eqdef \Axiom(x_1) \vee \bigdoublevee\limits_{n,t \in \Nat} \left(\left(x_1 = {n}\right) \wedge \bigdoublewedge\limits_{k \in \Nat} \bigdoublevee\limits_{p^\prime \prec p}  \Der_{p^\prime}\left({\Premise(n,t,k)}\right)\right).
	$$
	Here $\Axiom(x_1)$ is a $\Sigma_1$-formula saying that $x_1$ is a G\"odel number of an axiom of $\ACTMult$.
	
	\begin{proposition}\label{prop_Der}
		$\Der_{\pi(\alpha)}(x_1)$ is a $\Sigma_{\alpha \cdot 2 + 1}$-formula.
	\end{proposition}
	
	\begin{proof}
		Our goal is to find the computable function $F: \PN \to \Nat$ such that $F(\pi(\alpha))$ is an index of $\Der_{\pi(\alpha)}(x_1)$ as a $\Sigma_{\alpha\cdot 2 +1}$-formula. We shall use effective transfinite recursion. Firstly, $\Der_{\pi(0)}(x_1) = \Axiom(x_1)$ is $\Sigma_1$; let $F(\pi(0)) = \mathit{ax} = (\Sigma, \pi(1),1,e)$ be its index.
		
		To prove the induction step, assume that we know the index $e^{[<\alpha]}$ of the computable function $F^{[<\alpha]}$ such that $F^{[<\alpha]}(p^\prime)$ equals the index of the formula $\Der_{p^\prime}(x_1)$ if $p^\prime \prec \pi(\alpha)$, while otherwise it is undefined. Let $\chi_n$ be the index of the $\Sigma_0$-formula $(x_1 = n)$.
		\begin{enumerate}
			\item For $p^\prime \prec \pi(\alpha)$, given the index of $\Der_{p^\prime}(x_1)$ and $c \in \Nat$, one can use the function $\mathrm{sub}$ (Remark \ref{remark_manipulating_formulas}) to compute the index of $\Der_{p^\prime}\left({c}\right)$, which is a $\Sigma_{\alpha^\prime \cdot 2+1}$-formula for $\alpha^\prime = \pi^{-1}(p^\prime) < \alpha$. Therefore, the set of indices of formulas $\Der_{p^\prime}\left({\Premise(n,t,k)}\right)$ for $p^\prime \prec \pi(\alpha)$ is computable. Moreover, given $n$, $t$ and $k$, one can compute $e_1(n,t,k)$ such that $W_{e_1(n,t,k)}$ is the set of triples $(p^\prime, 0, e^\prime)$ where $e^\prime$ is the index of $\Der_{p^\prime}\left({\Premise(n,t,k)}\right)$. Consequently, $\bigdoublevee\limits_{p^\prime \prec \pi(\alpha)}  \Der_{p^\prime}\left({\Premise(n,t,k)}\right)$ is a $\Sigma_{\alpha\cdot 2 + 1}$-formula; its index is $I_1(n,t,k) = \langle \Sigma, \pi(\alpha \cdot 2 + 1), 0, e_1(n,t,k) \rangle$.
			\item Similarly, the function $e_2(n,t)$ is computable such that 
			$$
			W_{e_2(n,t)} = \{\langle \pi(0),0,\chi_n \rangle \} \cup \{\langle \pi(\alpha \cdot 2 + 1), 0, I_1(n,t,k) \rangle \mid k \in \Nat\}.
			$$
			Therefore, the formula $(x_1 = n) \wedge \bigdoublewedge\limits_{k \in \Nat} 
			\varphi^{\Sigma_{\alpha \cdot 2 + 1}}_{I_1(n,t,k),0}$ is $\Pi_{\alpha \cdot 2 + 2}$, and one can compute its index $I_2(n,t)$.
			\item Finally, one can find $e_3$ such that 
			$$
			W_{e_3} = \{\langle \pi(1), 0, ax \rangle \} \cup \{\langle \pi(\alpha \cdot 2 + 2),0,I_2(n,t) \rangle \mid n,t \in \Nat \}.
			$$
			Thus $\langle \Sigma, \pi(\alpha \cdot 2 + 3), 1, e_3\rangle$ is the index of $\Axiom(x_1) \vee \bigdoublevee\limits_{n,t \in \Nat} \varphi^{\Pi_{\alpha \cdot 2 + 2}}_{I_2(n,t),1}$ as desired.
		\end{enumerate}
		Summarizing the above reasonings, given $p = \pi(\alpha)$ and the index $e^{[<\alpha]}$, we can compute $e_3$. Let $\Psi(p,e^{[<\alpha]})$ be the index $e^{[\alpha]}$ of the function $F^{[\alpha]}$ such that $F^{[\alpha]}(p^\prime) = F^{[<\alpha]}(p^\prime)$ for $p^\prime \prec p$ and $F^{[\alpha]}(p) = e_3$ (otherwise, it is undefined). Applying effective transfinite recursion \cite[Theorem I.33]{Montalban23} we conclude that there is $e$ such that $\Phi_e(\pi(\alpha)) = \Psi(\pi(\alpha),e^{[<\alpha]})$ for all $\alpha < \omega^\omega$. It is not hard to see that $F(\pi(\alpha)) = \Phi_e(\pi(\alpha))$ is the index of $\Der_{\pi(\alpha)}(x_1)$ of the required rank.
	\end{proof}
	
	\begin{corollary}\label{corollary_ACTMult_to_SATOO}
		The derivability problem for $\ACTMult$ is one-one reducible to $\SATOO$.
	\end{corollary}
	\begin{proof}
		Take the function $F$ from the proof of Proposition \ref{prop_Der}. Given a sequent $\Pi \vdash C$, one can compute $r = \pi(\rho(\Pi \vdash C))$. Then $\langle F(r), \langle \ulcorner \Pi \vdash C \urcorner \rangle \rangle \in \SATOO$ if and only if $\Pi \vdash C$ is derivable in $\ACTMult$ (this can be proved by transfinite induction on the rank of $\Pi \vdash C$). The function that maps $\Pi \vdash C$ to $\langle F(r), \langle \ulcorner \Pi \vdash C \urcorner \rangle \rangle$ is injective.
	\end{proof}
	
	This completes the proof of Theorem \ref{th_main_1-equivalence}.
	
	\begin{proof}[Proof of Theorem \ref{th_main_1-equivalence}]
		Follows from Corollary \ref{corollary_SATOO_to_ACTMult} and Corollary \ref{corollary_ACTMult_to_SATOO} (and also from Myhill's isomorphism theorem \cite[\S7.4]{Rogers67}).
	\end{proof}
	
	\subsection{Closure Ordinal for $\ACTMult$}\label{ssec_closure_ordinal}
	
	Closure ordinal is a proof-theoretic complexity measure of a logic used in \cite{KuznetsovS22,KuznetsovS23} to study $\ACT$ with subexponentials and $\ACTMult$. The closure ordinal for a sequent calculus $\Logic$ is defined as follows. Let $\mathcal{D}_{\Logic}$ be a function that maps a set of sequents $S$ to the set $\mathcal{D}_{\Logic}(S)$, which is the union of $S$ with the set of sequents that can be obtained from elements of $S$ by one application of a rule of $\Logic$. The operator $\mathcal{D}_{\Logic}$ is called the \emph{immediate derivability operator for $\Logic$}. Let $T_{\Logic}(0) = \emptyset$, let $T_{\Logic}(\alpha+1) = \mathcal{D}_{\Logic}(T_{\Logic}(\alpha))$ and let $T_{\Logic}(\kappa) = \cup_{\alpha < \kappa} T_{\Logic}(\alpha)$ for $\kappa$ being limit. The closure ordinal for $\mathcal{D}_{\Logic}$ is the least $\alpha_0$ such that $T_{\Logic}(\alpha_0) = T_{\Logic}(\alpha_0+1)$. Clearly, $T_{\Logic}(\alpha_0)$ is the set of all sequents derivable in $\Logic$.
	
	The closure ordinal of $\Logic$ is related to the complexity of $\Logic$. Namely, it is proved in \cite[Theorem 6.2]{KuznetsovS23} that $T_{\Logic}(\alpha)$ is m-redicuble to $H(\alpha)^\prime$ (note that there Kleene's $\mathcal{O}$ is used). In what follows, the set $T_{\Logic}(\alpha_0)$ of derivable sequents is in $\Sigma^0_{\alpha_0}$. In \cite{KuznetsovS23}, the authors prove that the closure ordinal for $\ACTMult$ is not greater than $\omega^\omega$, thus the derivability problem for it has the upper bound $\Sigma^0_{\omega^\omega}$. Note that we have also proved this result in Corollary \ref{corollary_ACTMult_to_SATOO} using different methods. Finding the exact value of the closure ordinal for $\ACTMult$ was left as an open problem in \cite{KuznetsovS23}. Using Theorem \ref{th_main_complexity} we can easily solve it.
	
	\begin{corollary}
		The closure ordinal for $\ACTMult$ equals $\omega^\omega$.
	\end{corollary}
	
	Indeed, if the closure ordinal $\alpha_0$ for $\ACTMult$ was less than $\omega^\omega$, then the derivability problem for $\ACTMult$ would be in $\Sigma^0_{\alpha_0}$ while being $\Delta^0_{\omega^\omega}$-complete.
	
	Note that, even for $\ACTMult$ with the $(\mathrm{cut})$ rule, its closure ordinal equals $\omega^\omega$ due to the same argument. Proving this result for $\ACTMult+(\mathrm{cut})$ directly by analysing proofs in this calculus is probably a hard challenge.
	
	\subsection{Modification of the Construction for $\ACTMult^-$}\label{ssec_ACTMult_minus}
	
	In \cite{Kuznetsov21}, a fragment of infinitary action logic with the exponential is studied where a formula with Kleene star cannot be a subformula of a formula of the form $\bang A$. For example, $\bang (p \wedge q^\ast)$ is not allowed although $\bang p \wedge q^\ast$ and $(\bang p \wedge q)^\ast$ are well-formed formulas. It is proved in \cite{Kuznetsov21} that the derivability problem for that fragment is in $\Delta^1_1$ while derivability for the full logic with the exponential is $\Pi^1_1$-complete; thus, the restriction on Kleene star changes the complexity of infinitary action logic with the exponential. 
	
	One might be interested in whether this is also the case for $\ACTMult$ where one has the multiplexing modality instead of the exponential as in \cite{Kuznetsov21}. Let us denote the fragment of $\ACTMult$ where Kleene star is not allowed to be a subformula of a formula of the form $\bang A$ by $\ACTMult^-$. Note that the formula $\Energy_0$ contains both $\bang$ and $^\ast$ and $\Energy_n$ is defined using nested $\bang$-s; therefore, $\Energy_n$ is not a formula of $\ACTMult^-$. 
	
	It turns out that there is a simple modification of the main construction from Section \ref{ssec_construction} that belongs to $\ACTMult^-$. Let us present it. 
	\begin{definition}
		Let us define the formulas $\HEnergy_k$ and $\Killer$ as follows:
		\begin{enumerate}
			\item $\HEnergy_0 = \Energy_0$; $\HEnergy_{k+1} = \OKAY \wedge [ \energy ]^? \HEnergy_k^\ast $.
			
			\item $\Killer = \energy \bs \left((a_\Sigma \bs a_1^\ast \bs \okay) \wedge (a_\Pi \bs a_1^\ast \bs \okay)\right)$.
		\end{enumerate}
	\end{definition}
	Note that $\Energy_0$ is defined in the same way as in Definition \ref{def_Energy}. Now, we would like to prove the following lemma analogous to Lemma \ref{lemma_main}:
	\begin{lemma}\label{lemma_main_no_Kleene_under_bang}
		Let $\omega > h_1 \ge \dotsc \ge h_M$ and $\inp \in \Nat$. Consider the sequent
		\begin{equation}\label{eq_seq_main_no_Kleene_under_bang}
			a_L, a_1^{\inp}, a_X, \energy, \HEnergy_{h_M}, \dotsc, \HEnergy_{h_1}, \Killer \vdash a_L \cdot \okay.
		\end{equation}
		Let $\inp = \langle \langle X,\pi(\alpha),i,e \rangle,\langle n_1,\dotsc,n_i \rangle\rangle$.
		\begin{enumerate}
			\item If (\ref{eq_seq_main_no_Kleene_under_bang}) is derivable and $\omega^{h_1}+\dotsc+\omega^{h_M} > \alpha$, then $\inp \in \SATOO$.
			\item If $\inp \in \SATOO$, then (\ref{eq_seq_main_no_Kleene_under_bang}) is derivable.
		\end{enumerate}
	\end{lemma}
	
	Note that the inequality $\omega^{h_1}+\dotsc+\omega^{h_M} > \alpha$ is included in the first statement instead of the second one. This might seem wrong because this implies that (\ref{eq_seq_main_no_Kleene_under_bang}) is derivable if $\inp \in \SATOO$, even if $M = 0$, i.e. the sequent is of the form $a_L, a_1^{\inp}, a_X, \energy, \Killer \vdash a_L \cdot \okay$. This is, however, indeed true because of the formula $\Killer$. Namely, the sequent $a_L, a_1^{\inp}, a_X, \energy, \Killer \vdash a_L \cdot \okay$ is derivable independently of $\inp$:
	$$
		\infer[(\bs L)]
		{
			a_L, a_1^{\inp}, a_\Sigma, \energy, \Killer \vdash a_L \cdot \okay
		}
		{
			\infer[(\wedge L_1)]
			{
				a_L, a_1^{\inp}, a_\Sigma, (a_\Sigma \bs a_1^\ast \bs \okay) \wedge (a_\Pi \bs a_1^\ast \bs \okay) \vdash a_L \cdot \okay
			}
			{
				\infer[(\bs L)]
				{
					a_L, a_1^{\inp}, a_\Sigma, a_\Sigma \bs a_1^\ast \bs \okay \vdash a_L \cdot \okay
				}
				{
					\infer[(\bs L)]
					{
						a_L, a_1^{\inp}, a_1^\ast \bs \okay \vdash a_L \cdot \okay
					}
					{
						a_L, \okay \vdash a_L \cdot \okay
						&
						a_1^{\inp} \vdash a_1^\ast 
					}
					&
					a_\Sigma \vdash a_\Sigma
				}
			}
			&
			\energy \vdash \energy
		}
	$$
		
	\begin{proof}[Proof of Lemma \ref{lemma_main_no_Kleene_under_bang}]
		Again, the proof is by induction on $\beta = \omega^{h_1} + \dotsc + \omega^{h_M}$. Note that $\HEnergy_k$ and $\Killer$ are locked formulas so one can apply the bottom-top analysis in the same way as in the proof of Lemma \ref{lemma_main}. The base case holds, because the sequent $a_L, a_1^{\inp}, a_X, \energy, \Killer \vdash a_L \cdot \okay$ is derivable. Now, let us prove the induction step.
		
		\textbf{Case 1.} Let $h_M > 0$ (compare it with Case 1 in the proof of Lemma \ref{lemma_main}). Then $\HEnergy_{h_M} = \OKAY \wedge [ \energy ]^? \HEnergy_{h_M-1}^\ast$. Let us apply the bottom-top analysis.
		
		\begin{center}
			\begin{tabular}{ll}
				&\\[-11pt]
				$a_L, a_1^{\inp}, a_X, \energy, \principal{\HEnergy_{h_M}}, \dotsc, \HEnergy_{h_1} \vdash a_L \cdot \okay$
				&
				\ref{item_bottom-top_wedge} $(\wedge)$
				\\[2pt]
				&\\[-11pt]
				$a_L, a_1^{\inp}, a_X, \energy, \principal{[ \energy ]^? \HEnergy_{h_M-1}^\ast}, \HEnergy_{h_{M-1}}, \dotsc, \HEnergy_{h_1} \vdash a_L \cdot \okay$
				&
				\ref{item_bottom-top_bs} $(\bs)$
				\\[2pt]
				&\\[-11pt]
				$a_L, a_1^{\inp}, a_X, \principal{\energy \cdot \HEnergy_{h_M-1}^\ast}, \HEnergy_{h_{M-1}}, \dotsc, \HEnergy_{h_1} \vdash a_L \cdot \okay$
				&
				\ref{item_bottom-top_cdot} $(\cdot)$
				\\[2pt]
				&\\[-11pt]
				$a_L, a_1^{\inp}, a_X, \energy, \principal{\HEnergy_{h_M-1}^\ast}, \HEnergy_{h_{M-1}}, \dotsc, \HEnergy_{h_1} \vdash a_L \cdot \okay$
				&
				\ref{item_bottom-top_bang} $(\bang)$
				\\[2pt]	
			\end{tabular}
			\begin{equation}\label{eq_seq_BTA1_NEW}
				\mbox{for each $l$,} \quad a_L, a_1^{\inp}, a_X, \energy, \HEnergy_{h_M-1}^{l}, \HEnergy_{h_{M-1}}, \dotsc, \HEnergy_{h_1} \vdash a_L \cdot \okay
			\end{equation}
		\end{center}
		
		We can apply the induction hypothesis to (\ref{eq_seq_BTA1_NEW}). Let $\beta^\prime_l \eqdef \omega^{h_1}+\dotsc+\omega^{h_{M-1}}+\omega^{h_M-1}\cdot l$. Note that $\beta^\prime_l < \beta$. By the induction hypothesis:
		\begin{enumerate}
			\item if (\ref{eq_seq_BTA1_NEW}) is derivable and $\beta^\prime_l > \alpha$, then $\inp \in \SATOO$;
			\item if $\inp \in \SATOO$, then (\ref{eq_seq_BTA1_NEW}) is derivable.
		\end{enumerate}
		Let us prove both statements of the lemma. Assume that (\ref{eq_seq_main}) is derivable and $\beta > \alpha$. Then (\ref{eq_seq_BTA1_NEW}) is derivable for each $l \in \Nat$. Since $\beta > \alpha$ and $\beta = \sup_{l \in \Nat} \beta^\prime_l$, then it holds that $\beta^\prime_l > \alpha$ for some $l$. In what follows, $\inp \in \SATOO$ as desired. Conversely, let $\inp \in \SATOO$. In what follows, (\ref{eq_seq_BTA1_NEW}) is derivable for each $l \in \Nat$ and hence so is (\ref{eq_seq_main}).

		\textbf{Case 2.} If $h_M = 0$, then $\HEnergy_{h_M} = \HEnergy_0 = \Energy_{0}$, so the reasonings are the same as in Case 2 in the proof of Lemma \ref{lemma_main} (the only difference is where to take the inequality $\beta > \alpha$ into account). 
	\end{proof}
	
	\begin{corollary}
		The derivability problem for $\ACTMult^-$ is recursively isomorphic to $\SATOO$ and hence is $\Delta^0_{\omega^\omega}$-complete under Turing reductions.
	\end{corollary}

	\subsection{Fragments of $\ACTMult$ with Intermediate Complexities}\label{ssec_family_of_fragments}
	
	As we have proved, the fragment $\ACTMult^-$ has the same complexity as $\ACTMult$. One might be interested in whether one can restrict $\ACTMult$ in a natural way to obtain a logic with complexity strictly between $\Delta^0_\omega$ and $\Delta^0_{\omega^\omega}$. The answer to this question is positive. Namely, let us consider the fragment of $\ACTMult$ obtained by considering only formulas such that their $\{\ast,\bang\}$-depth is not greater than a fixed $k \in \Nat$. The notion of $\{\ast,\bang\}$-depth is similar to that of modal depth for modal logic. One can easily define it using the rank function $\rho$:
	
	\begin{definition}
		$\{\ast,\bang\}$-depth of a formula $A$ is $k \in \Nat$ such that $\omega^k \le \rho(A) < \omega^{k+1}$.
	\end{definition}
	
	\begin{example}
		$\{\ast,\bang\}$-depth of the formulas $\bang(p \wedge q^\ast)$ and $(\bang p \wedge q)^\ast$ equals $2$; $\{\ast,\bang\}$-depth of $\bang p \wedge q^\ast$ equals $1$.
	\end{example}
	
	Let $\Fm^{(k)}$ be the set of formulas of $\{\ast,\bang\}$-depth not greater than $k$. Let $\ACTMult^{(k)}$ be the fragment of $\ACTMult$ where only formulas from $\Fm^{(k)}$ are used.
	
	\begin{theorem}
		For $k > 0$, derivability in $\ACTMult^{(k)}$ is $\Delta^0_{\omega^k}$-hard and it is in $\Delta^0_{\omega^{k+1}}$.
	\end{theorem}
	
	\begin{proof}
		The function $\mathrm{seq}$ from Corollary \ref{corollary_SATOO_to_ACTMult} reduces $\SAT_{\omega^k}$ to derivability in $\ACTMult^{(k)}$. Indeed, if $\inp = \langle \langle X,\pi(\alpha),i,e \rangle,\langle n_1,\dotsc,n_i \rangle\rangle$ and $\alpha < \omega^{k}$, then 
		$$
		\mathrm{seq}(\inp) = a_L, a_1^{\inp}, a_X, \energy, \Energy_{h_M}, \dotsc, \Energy_{h_1} \vdash a_L \cdot \okay
		$$
		where $\omega^{h_i} \le \alpha + 1$, hence $h_i < k$. Note that the $\{\ast,\bang\}$-depth of $\Energy_h$ equals $h+1$; then, the $\{\ast,\bang\}$-depth of $\Energy_{h_i}$ is not greater than $k$, as desired.
		
		Conversely, let us use the reduction from Corollary \ref{corollary_ACTMult_to_SATOO}. Given the sequent $\Pi \vdash C$ of $\ACTMult^{(k)}$ with the rank $\alpha = \rho(\Pi \vdash C)$, compute $\langle F(\pi(\alpha)), \langle \ulcorner \Pi \vdash C \urcorner \rangle \rangle$. It remains to notice that, since $\alpha < \omega^{k+1}$, it holds that $2 \cdot \alpha + 1 < \omega^{k+1}$; therefore, $\langle F(\pi(\alpha)), \langle \ulcorner \Pi \vdash C \urcorner \rangle \rangle \in \SATOO$ is equivalent to $\langle F(\pi(\alpha)), \langle \ulcorner \Pi \vdash C \urcorner \rangle \rangle \in \SAT_{\omega^{k+1}}$.
	\end{proof}
	Consequently, the complexities of the logics $\ACTMult^{(k)}$ tend to the $\omega^\omega$ level of the hyperarithmetical hierarchy. 
	
	\section{Conclusion}\label{sec_conclusion}
	
	We have proved that the complexity of the derivability problem for $\ACTMult$ is $\Delta^0_{\omega^\omega}$-complete under Turing reductions and that the closure ordinal for this calculus equals $\omega^\omega$ (even if the cut rule is included in the list of rules). We have shown that the set of sequents derivable in $\ACTMult$ is recursively isomorphic to the satisfaction predicate $\SATOO$ and is m-equivalent to the set $H(\omega^\omega)$, which is defined using the polynomial notation on ordinals below $\omega^\omega$. Therefore, the logic $\ACTMult$ has strictly hyperarithmetical complexity. The same holds for the logic $\ACTMult^-$.
	
	Using the same reductions we have shown that the $k$-th fragment $\ACTMult^{(k)}$ of infinitary action logic with multiplexing has the complexity between $\Delta^0_{\omega^k}$ and $\Delta^0_{\omega^{k+1}}$; in what follows, the logics $\ACTMult^{(2k)}$ have strictly increasing complexities that tend to $\Delta^0_{\omega^\omega}$. Establishing the precise complexity of $\ACTMult^{(k)}$ is quite an interesting question for further research. Our conjecture is that $\ACTMult^{(k)}$ is $\Delta^0_{\omega^{k}}$-complete, which perhaps could be proved by refining the formula $\Der_p(x_1)$ introduced in Section \ref{ssec_ACTMult_to_SATOO} and thus decreasing its rank. Namely, taking into account that some of the disjunctions and conjunctions are actually finitary could be helpful.

	\section*{Acknowledgments}
	
	I am grateful to Stepan L. Kuznetsov and Stanislav Speranski for their careful attention to my work and valuable advice.
	
	\section*{Funding}
	
	This work was performed at the Steklov International Mathematical Center and supported by the Ministry of Science and Higher Education of the Russian Federation (agreement no. 075-15-2022-265).
	
	\bibliographystyle{plain}
	\bibliography{ACT_Multiplexing_complexity}

\end{document}